\documentclass[11pt]{amsart}

\usepackage{amsmath,amssymb,amsthm,multirow,hyperref,enumitem,tikz}
\usepackage{fullpage}

\usetikzlibrary{decorations.pathmorphing,decorations.pathreplacing,calc}

\newtheorem{lemma}{Lemma}[section]
\newtheorem{corollary}[lemma]{Corollary}
\newtheorem{question}[lemma]{Question}

\newtheorem{theorem}[lemma]{Theorem}
\newtheorem{example}[lemma]{Example}
\begin{document}

\title[Orbit dimension]{Some remarks on the orbit dimension of transitive groups and on the metric dimension of Johnson graphs}  

\author[A.~Drera]{Alice Drera}
\address{Scuola Galileiana di Studi Superiori, Universit\`a di Padova; Dipartimento di Matematica ``Tullio Levi-Civita'', Universit\`a di Padova,  Via Trieste 63, 35121 Padova, Italy}
\email{alice.drera@studenti.unipd.it}

\author[P.~Spiga]{Pablo Spiga}
\address{Dipartimento di Matematica e Applicazioni, University of Milano-Bicocca, Via Cozzi 55, 20125 Milano, Italy} 
\email{pablo.spiga@unimib.it}
\begin{abstract}The orbit dimension $\sigma(G)$ (also called the separation number or rigidity index) of a permutation group $G$ with domain $\Omega$ is the minimum cardinality of a subset $S \subseteq \Omega$ such that, for any two distinct elements $\omega,\omega'\in \Omega$, there exists $\alpha\in S$ for which $\omega$ and $\omega'$ lie in distinct orbits of the stabilizer $G_\alpha$.

In this paper, we first observe that if $G$ is transitive, then
$\sigma(G)\le |\Omega|-r+1$,
where $r$ is the rank of $G$, and we obtain strong structural information on the groups for which equality holds.

Next, we investigate the orbit dimension in the case where $G$ is the symmetric group of degree $n$, acting on the set of $k$-subsets of $\{1,\ldots,n\}$. In this case, this invariant equals the metric dimension of Johnson graphs.
\end{abstract}
\maketitle

\section{Introduction}
Let $G$ be a permutation group on a finite set $\Omega$.
Given $\omega,\omega',\alpha\in \Omega$, we say that $\alpha$
\textit{\textbf{separates}} $\omega,\omega'$ if $\omega$ and $\omega'$
belong to distinct orbits of the stabilizer $G_\alpha$ of $\alpha$.
A subset $S$ of $\Omega$ is said to be \textit{\textbf{separating}} for $G$
if, for any two distinct elements $\omega,\omega'\in \Omega$, there exists
$\alpha\in S$ separating $\omega,\omega'$.
We say that $S$ is an \textit{\textbf{irredundant separating}} set if $S$ is
separating and if $S\setminus\{\alpha\}$ is not separating for each
$\alpha\in S$.
The minimum cardinality of a separating set for $G$ is said to be the
\textit{\textbf{orbit dimension}} of $G$, and it is denoted by $\sigma(G)$.

As far as we are aware, this invariant has appeared in the literature only a few times: in an unpublished note by Cameron~\cite{cameron}, which was later extended in a journal article in~\cite{bayley} with Bailey, and in two papers by Fijav\v{z} and Mohar~\cite{FF,F}.  
Cameron's note contains several interesting remarks on the orbit dimension and its relationship with the base size of a permutation group, as well as some conjectures aimed at extending classical results on base sizes to this new invariant.  In the unpublished note, the orbit dimension was named separation number.

The point of view adopted by Fijav\v{z} and Mohar (who use the term \textit{\textbf{rigidity index}}) is graph-theoretic. In fact, they study the separation number when $G$ is the automorphism group of a graph, motivated by connections between separating sets and colouring problems. This perspective is very natural, since the orbit dimension (a.k.a. separation number) is closely related to the \textit{\textbf{metric dimension}} of a graph. Recall that the metric dimension of a connected graph $\Gamma$ is the minimum cardinality of a set $S$ of vertices such that, for any pair of vertices $u,w$, there exists $s\in S$ with $d(u,s)\ne d(w,s)$,
where $d(x,y)$ denotes the distance between the vertices $x$ and $y$. For distance-transitive graphs, the orbit dimension (a.k.a. rigidity index) equals the metric dimension.
For further work on the metric dimension, see for instance~\cite{bayley,bayleyy}. Because of the ambiguity in the literature on the name of this invariant, we use orbit dimension when dealing with general permutation groups and we switch to metric dimension (which is more widely used) when possible.

Independently of whether one adopts a permutation-group-theoretic or a graph-theoretic point of view, we find the orbit dimension to be a particularly interesting invariant, both because of the difficulty and complexity of obtaining structural information about it and because of its close relationship with the lattice of partitions (see Section~\ref{sec:2} for details).

The scope of this paper is twofold. First, we re-establish a bound relating the rank of a transitive permutation group to its separation number, and we characterize the groups attaining this bound.

\begin{theorem}\label{l:3}
Let $G$ be a transitive group of degree $n$ and rank $r$. Then $\sigma(G)\le n-r+1$. Moreover,
if $\sigma(G) = n - r + 1$ and $G$ is not regular, then $G$ admits a system of imprimitivity $\Sigma$ consisting of
$r-1$ blocks such that the setwise
stabilizer of each block acts $2$-transitively on the block and the induced action of $G$
on $\Sigma$ is regular.
\end{theorem}

Second, in order to build some intuition for the orbit dimension in the case of primitive groups, and in particular in relation to the conjectures proposed in~\cite{cameron} between orbit dimension and base size, we investigate the orbit dimension of the symmetric group acting on $k$-subsets. In this context, we let $\sigma(m,k)$ denote the orbit dimension of $\mathrm{Sym}(m)$ in its action on the $k$-subsets of $\{1,\ldots,m\}$. Observe that in this context $\sigma(m,k)$ equals the metric dimension of the \textbf{\textit{Johnson graph}}. The vertices of the Johnson graph $J(m,k)$ are the $k$-subsets of an $m$-set; two vertices are adjacent when their intersection contains $k-1$ elements. The metric dimension of Johnson graphs has already been studied in~\cite{baileyyy}, focusing on obtaining upper bounds on $\sigma(m,k)$. Here we improve these upper bounds and also obtain some strong lower bounds. 

\begin{theorem}\label{thrm:main}
Let $m$ be a positive integer and let $G=\mathrm{Sym}(m)$ be the symmetric group of degree $m$ in its action on the $k$-subsets of $\{1,\ldots,m\}$, where $k\ge 3$. Let 
\[
k'=\frac{4\big((k-1)^{2k}-1\big)}{k-2},
\]
and let $q,r\in\mathbb{N}$ be such that $m=k'q+r$ with $r<k'$. Then
\[
\frac{6m-8}{3k+1}\le \sigma(m,k)\le \frac{2k'q}{k}+
\left(
\left\lfloor\frac{k}{2}\right\rfloor+1\right)r\le \frac{2m}{k}+\left(\left\lfloor\frac{k}{2}\right\rfloor+1\right)(k'-1).
\]
\end{theorem}

We record some remarks concerning Theorem~\ref{thrm:main}.
\begin{itemize}
\item We exclude the cases $k\in\{1,2\}$, since $\sigma(m,1)$ and $\sigma(m,2)$ are known explicitly (see~\cite{bayley,cameron}).

\item The upper bound improves asymptotically on the bound $\lfloor k(m+1)/(k+1)\rfloor$ obtained in~\cite{baileyyy}.

\item The lower bound is stronger than the bound obtained by estimating $\sigma(m,k)$ via the base size of $\mathrm{Sym}(m)$ acting on $k$-subsets using the results of Halasi~\cite{halasi}, which yields $\lfloor 2(m-1)/(k+1)\rfloor$.
\item Theorem~\ref{thrm:main} implies that $\sigma(m,k)/m \to 0$ as $k\to\infty$, a feature that does not follow directly from the existing literature.
\end{itemize}

Determining an explicit formula for $\sigma(m,k)$ appears to be difficult.  We refine Theorem~\ref{thrm:main} in the case $k=3$.
\begin{theorem}\label{thrm:main2}
Let $m$ be a positive integer and let $q,r\in\mathbb{N}$ with $m=13q+r$ and $0\le r<13$. Then
\[
\frac{36m-50}{59}\le \sigma(m,3)\le 8q+2r\le \frac{8m}{13}+24.
\]
\end{theorem}
The lower bound in Theorem~\ref{thrm:main2} should not be taken too seriously. Its purpose is rather to show that the general lower bound in Theorem~\ref{thrm:main} is not optimal. Indeed, for $k=3$, the leading coefficient in Theorem~\ref{thrm:main} is $6/(3k+1)=3/5$, whereas here we obtain the larger value $36/59$.
We conjecture that
$\lim_{m\to \infty}\sigma(m,3)/m=8/13$.

\section{Proof of Theorem~\ref{l:3}}\label{sec:2}
Separating sets can be phrased in terms of partitions: this is classic and has been explored in all papers mentioned in the introduction; we only repeat the essential parts to characterize the groups attaining the bound in Theorem~\ref{thrm:main} (which is the only part we claim originality in this section).

The collection of all partitions of $\Omega$ is a partially ordered set.
Indeed, given two partitions $\pi$ and $\pi'$, we have $\pi\le \pi'$ if,
for every $P\in \pi$, there exists $P'\in \pi'$ with $P\subseteq P'$,
that is, $\pi$ is a \textit{\textbf{refinement}} of $\pi'$.
Actually, this partially ordered set is also a lattice.
Indeed, given two partitions $\pi,\pi'$ of $\Omega$, the meet
$\pi\wedge\pi'$ is the partition of $\Omega$ given by
\[
\{P\cap P'\mid P\in \pi,\ P'\in \pi' \text{ with } P\cap P'\ne\emptyset\}.
\]
The join $\pi\vee\pi'$ of $\pi$ and $\pi'$ can be defined analogously. The bottom element in this lattice is the partition into singletons $\{\{\omega\}\mid \omega\in \Omega\}$ and the top element is the universal partition $\{\Omega\}$.

A separating set $S$ for $G$ has a natural interpretation within the
lattice of partitions of $\Omega$. 
For each $\omega\in \Omega$, let $\pi_\omega$ be the partition of $\Omega$
given by the orbits of $G_\omega$. 

\begin{lemma}\label{l:1}
The set $S$ is separating for $G$ if and only if the meet of all the
partitions $\{\pi_\alpha\mid \alpha\in S\}$ is the partition of $\Omega$
into singletons.
\end{lemma}

\begin{proof}
Assume $S$ is separating.
If the meet of the partitions $\{\pi_\alpha\mid \alpha\in S\}$ is not
the partition of $\Omega$ into singletons, then there exist two distinct
elements $\omega,\omega'\in\Omega$ belonging to a part of $\pi_\alpha$
for every $\alpha\in S$.
However, this means that no element of $S$ separates $\omega$ and
$\omega'$.
The converse is also immediate.
\end{proof}
Here are some natural observations that one can make, pivoting on Lemma~\ref{l:1}, when $G$ is
transitive.
Let $r$ be the \textit{\textbf{rank}} of $G$, that is, the number of
orbits of $G_\alpha$ on $\Omega$. Analogously, $r$ is the number of orbits of $G$ in its action on the cartesian product $\Omega\times\Omega$. In particular, when $r=2$, $G$ acts transitively on the set of ordered pairs of distinct elements of $\Omega$, that is, $G$ is $2$-transitive.

\begin{lemma}\label{l:2}
Let $G$ be a transitive permutation group on $\Omega$, let $r$ be its
rank and let $S$ be an irredundant separating set. Then, $|S|\le |\Omega|-r+1$. In particular, $\sigma(G)\le |\Omega|-r+1$.
\end{lemma}

\begin{proof}
Let $S=\{\alpha_1,\ldots,\alpha_s\}$ be an irredundant separating set of cardinality
$s$.
As $G$ has rank $r$, for each $\omega\in \Omega$, the partition $\pi_\omega$ has $r$ parts.
We prove that, for each $\ell\in \{1,\ldots,r\}$, the partition
$\pi_1\wedge\pi_2\wedge\cdots\wedge \pi_\ell$ has at least $r+\ell-1$ parts.

When $\ell=1$, this is clear from the definition of $r$.
Assume then $\ell>1$ and, arguing by contradiction, assume that
$\pi_1\wedge\pi_2\wedge\cdots\wedge \pi_\ell$ has at most $r+\ell-2$ parts.
Since by induction
$\pi_1\wedge\pi_2\wedge\cdots\wedge \pi_{\ell-1}$ has at least $r+\ell-2$ parts,
we deduce
$\pi_1\wedge\pi_2\wedge\cdots \wedge \pi_{\ell-1}\le \pi_\ell$
and hence
\[
\pi_1\wedge\pi_2\wedge\cdots \wedge \pi_{\ell-1}\wedge\pi_{\ell+1}\wedge\cdots\wedge\pi_s
=\bigwedge_{i=1}^s
\pi_i.
\]
As $S$ is separating, by Lemma~\ref{l:1}, the partition above is the partition into singletons.
Applying again Lemma~\ref{l:1}, we also deduce that
$S\setminus\{\alpha_\ell\}$ is separating, which contradicts the fact
that $S$ is irredundant.

From the paragraph above, the partition
$\pi_1\wedge\cdots\wedge\pi_s$ has at least $r+s-1$ parts and hence $|\Omega|\ge r+s-1$.

Finally, observe that a separating set of minimal cardinality is irredundant.
\end{proof}

Separating sets are also related to an important concept in permutation groups. A \textit{\textbf{base}} for a permutation group $G$ is a sequence $(\omega_1 , \ldots , \omega_\ell )$ whose pointwise stabilizer in $G$ is the identity. A base is
said to be \textit{\textbf{irredundant}} if no base point is fixed by the stabilizer of its predecessors. Therefore, this gives
rise to a decreasing sequence of stabilizers
$$
G > G_{\omega_1} > G_{\omega_1 ,\omega_2} > \cdots > 
G_{\omega_1 ,\omega_2 ,\ldots,\omega_{\ell-1}} > 
G_{\omega_1,\omega_2,\ldots,\omega_\ell}= 1.
$$

The minimal cardinality of a base, called the \textit{\textbf{base size}} of $G$, is usually denoted by $b(G)$ and has played a key role in
the investigation of primitive groups.
\begin{lemma}\label{l:0}
If $S$ is separating for $G$, then $S$ is a base. In particular, $b(G)\le \sigma(G)$.
\end{lemma}
\begin{proof}
Let $S$ be separating.
If $S$ is not a base, then there exists a non-identity element $g\in G$ fixing $S$ pointwise. In particular, there exists $\omega\in \Omega$ with $\omega'=\omega^g\ne \omega$. Clearly, no element of $S$ separates $\omega,\omega'$.
\end{proof}

Cameron and Fon-der-Flaass~\cite{CF} (see also~\cite[Section~4.14]{Peter}) have proved that in a finite permutation group the following conditions are equivalent:
\begin{itemize}
\item all irredundant bases have the same size;
\item the irredundant bases are invariant under re-ordering;
\item the irredundant bases are the bases of a matroid.
\end{itemize}
A permutation group satisfying one, and hence all, of these conditions is said to be \textit{\textbf{IBIS}}, that is, \textit{\textbf{Irredundant Bases of Invariant Size}}. In particular, in an IBIS group $G$, all irredundant bases have cardinality $b(G)$ and this value is sometimes referred to as the \textit{\textbf{rank}} of the IBIS group, since it is the rank of the corresponding matroid. Motivated by a question of Cameron~\cite[Section~4.14]{Peter}, some effort has been put into classifying the primitive IBIS groups. Currently, the classification is reduced to almost simple groups (of rank $3,4,5,6$) and to affine groups. Here we would like to propose an analogue problem.
\begin{question}\label{eq:1}
{\rm What are the primitive groups in which all irredundant separating sets have the same cardinality? }
\end{question}

After this short digression, we give some consequences of Lemma~\ref{l:2}.
\begin{corollary}\label{cor}
Let $G$ be a transitive group of degree $n\ge 2$.
Then $\sigma(G)=n-1$ if and only if $G$ is $2$-transitive. In particular, in a $2$-transitive group all irredundant separating sets have the same cardinality.
\end{corollary}

\begin{proof}Let $r$ be the rank of $G$.
If $\sigma(G)=n-1$, then Lemma~\ref{l:2} implies $n-1\le n-r+1$ and hence
$r\le 2$.
As $n\ge 2$, we have $r=2$ and hence $G$ is $2$-transitive.

If $G$ is $2$-transitive, then for every $\omega\in \Omega$, $\pi_\omega$
is the partition $\{\{\omega\},\Omega\setminus\{\omega\}\}$.
It is clear that we need the meet of $n-1$ such partitions to reach the
partition into singletons in the partition lattice, and hence the result follows from Lemma~\ref{l:1}.
\end{proof}

With some work, we can strengthen the previous corollary as in the statement of Theorem~\ref{l:3}.
\begin{proof}[Proof of Theorem~$\ref{l:3}$]
From Lemma~\ref{l:2}, we have $\sigma(G)\le n-r+1$ and hence we assume that $\sigma(G) = n - r + 1$.

Let $\Omega$ be the domain of $G$. 
By Lemmas~\ref{l:1} and~\ref{l:2}, for any two elements $\omega,\omega' \in \Omega$ with
$G_\omega \neq G_{\omega'}$, the meet $\pi_\omega \wedge \pi_{\omega'}$ consists of exactly
$r+1$ parts. In other words, when we intersect the orbits of $G_\omega$ and $G_{\omega'}$,
only one orbit can be split while all others coincide; otherwise, we would have
$\sigma(G) < n - r + 1$.  
We refer to this property as $(\dagger)$.

If $r = 2$, that is, if $G$ is $2$-transitive, then by
Corollary~\ref{cor} the conclusion follows by taking $\Sigma = \{\Omega\}$.
Hence, for the remainder of the proof we assume $3 \le r \le n - 1$,
so that $G$ is neither regular nor $2$-transitive.

Fix $\omega \in \Omega$. For each $i \in \{1,\ldots,n\}$, let $m_i$ denote the number of
$G_\omega$-orbits of cardinality $i$, and label them
\[
O_{i,1}, \ldots, O_{i,m_i},
\]
where $O_{1,1} = \{\omega\}$. Clearly, $n = \sum_{i=1}^n i m_i$.

\smallskip
We claim that $G_\omega$ fixes only the point $\omega$, that is, $m_1 = 1$.
Suppose, for a contradiction, that $G_\omega$ fixes more than one point.
Let $\alpha \in \Omega$ lie in a $G_\omega$-orbit $O$ of maximal cardinality.
Since $r < n$, we have $|O| \ge 2$.
By hypothesis, $G_\alpha$ fixes a point $\alpha' \neq \alpha$.
As $G_\alpha = G_{\alpha'}$ and $G_\omega \neq G_\alpha$, it follows that
$\alpha'$ lies in a $G_\omega$-orbit of size greater than $1$.
Thus the meet $\pi_\omega \wedge \pi_\alpha$ has at least $r+2$ parts, since the part of $\pi_\omega$ containing $\alpha$ and the part of $\pi_\alpha$ containing $\omega$ are both split into at least two parts,
contradicting $(\dagger)$.  
Therefore, $G_\omega$ fixes only $\omega$. In particular, $m_1 = 1$.

\smallskip

Let $i \ge 2$ with $m_i \neq 0$ and let $\omega' \in O_{i,\ell}$.
Property $(\dagger)$ implies
\[
\pi_\omega \wedge \pi_{\omega'}
=\left\{
O_{j,\kappa}\mid (j,\kappa) \neq (i,\ell)\right\}
\;\cup\;
\{\{\omega'\},
O_{i,\ell} \setminus \{\omega'\}\}.
\]
Relative to $\pi_\omega$, this increases the number of parts of size $1$ and $i-1$ by $1$,
and decreases the number of parts of size $i$ by $1$.
Since $\pi_{\omega'}$ has the same shape as $\pi_\omega$ and $\pi_\omega$ has only one
singleton part, one of the following must occur:

\begin{enumerate}
\item\label{casE1}
The parts of $\pi_{\omega'}$ are $\{\omega'\}$, $(O_{i,\ell} \setminus \{\omega'\}) \cup \{\omega\}$,
and the remaining $r-2$ parts coincide with those of $\pi_\omega$.
\item\label{casE2}
There exists $\ell' \in \{1,\ldots,m_{i-1}\}$ such that the parts of $\pi_{\omega'}$ are
$\{\omega'\}$, $O_{i-1,\ell'} \cup \{\omega\}$, $O_{i,\ell} \setminus \{\omega'\}$,
and the remaining $r-3$ parts coincide with those of $\pi_\omega$.
\end{enumerate}
Roughly speaking, in~\eqref{casE1}, the new part of size $i$ of $\pi_{\omega'}$ is $O_{i,\ell}$ with $\omega'$ replaced by $\omega$. In~\eqref{casE2}, the new part of size $i$ of $\pi_{\omega'}$ is obtained by joining $\omega$ to a part of $\pi_\omega$  of size $i-1$.

We say that $O_{i,\ell}$ is of type~1 if~\eqref{casE1} holds, and of type~2
if~\eqref{casE2} holds.  
Note that type~2 cannot occur when $i = 2$, since $m_1 = 1$.

Let $i_0 \ge 2$ be minimal with $m_{i_0} \neq 0$.  
The orbits of cardinality $i_0$ are all of type $1$, because a necessary condition for an orbit of cardinality $i$ to be of type~2 is that $m_{i-1}\ne 0$.

\smallskip

If $O_{i,\ell}$ is of type~1 and $\omega' \in O_{i,\ell}$, then
\[
\Delta = O_{i,\ell} \cup \{\omega\}
\]
is invariant under both $G_\omega$ and $G_{\omega'}$.
Thus $N = \langle G_\omega, G_{\omega'} \rangle$ acts transitively on $\Delta$,
and hence $\Delta$ is a block of imprimitivity for the action of $G$ on $\Omega$.
Moreover, $N_\omega = G_\omega$ has orbits $\{\omega\}$ and $\Delta \setminus \{\omega\}$,
so $N$ acts $2$-transitively on $\Delta$.

\smallskip

We claim that there is a unique $G_\omega$-orbit of type~$1$.  
Suppose, for a contradiction, that $O_{i_0,1}$ and $O_{i,\ell}$ are both of type~$1$, where $i\ge i_0$.
Choose $\omega'\in O_{i_0,1}$ and $\omega''\in O_{i,\ell}$, and set
\[
\Lambda=\{\omega\}\cup O_{i_0,1}\cup O_{i,\ell}.
\]
As in the discussion preceding this claim, the fact that $O_{i_0,1}$ is of type~$1$ implies
that $\Delta_0=\{\omega\}\cup O_{i_0,1}$ is invariant under both $G_\omega$ and $G_{\omega'}$,
and hence $\Delta_0$ is a block of imprimitivity for the action of
$\langle G_\omega,G_{\omega'}\rangle$ on $\Omega$. In particular, $\Delta_0$ is also a block
for any subgroup containing $\langle G_\omega,G_{\omega'}\rangle$.

Similarly, since $O_{i,\ell}$ is of type~$1$, the set 
$\Delta_1=\{\omega\}\cup O_{i,\ell}$
is invariant under both $G_\omega$ and $G_{\omega''}$, and hence it is a block of imprimitivity
for $\langle G_\omega,G_{\omega''}\rangle$; therefore $\Delta_1$ is also a block for any subgroup
containing $\langle G_\omega,G_{\omega''}\rangle$.

Now let
\[
M=\langle G_\omega,G_{\omega'},G_{\omega''}\rangle.
\]
By the type~$1$ condition (applied twice) and the transitivity argument already used above,
the group $M$ acts transitively on $\Lambda$. 
Since $\Delta_0$ is a block for $M$ and $\Delta_0\subseteq \Lambda$, the set $\Lambda$ is a union of
$M$-translates of $\Delta_0$.
Thus $1+i_0$ divides $|\Lambda|=1+i_0+i$.
Likewise, $|\Delta_1|=1+i$ divides $1+i_0+i$.
But these divisibilities force $1+i_0 \mid i$ and $1+i \mid i_0$, which is impossible for positive
integers $i_0,i$ with $i\ge i_0\ge 2$. This contradiction shows that there is exactly one orbit of
type~$1$.

In particular, since  all $G_\omega$-orbits of size
$i_0$ are of type~$1$, we have $m_{i_0}=1$, and every $G_\omega$-orbit of size
strictly larger than $i_0$ must be of type~$2$. Finally, note that such larger orbits do exist,
since $G$ is not $2$-transitive.

\smallskip

If $i > i_0$ and $m_{i-1} = 0$, then $m_i = 0$.
Otherwise, an orbit of size $i$ would have to be of type~2, requiring $m_{i-1} \neq 0$.
This implies $m_{i_0+1}\neq 0$, because otherwise $G_\omega$ would have only two orbits, turning $G$ into a $2$-transitive group, which is a contradiction. 
We now show that $m_{i_0+2}=0$.  
Suppose, for a contradiction, that $m_{i_0+2}\neq 0$.  
Choose
$\omega'\in O_{i_0,1}$, 
$\omega''\in O_{i_0+1,1}$, 
$\omega'''\in O_{i_0+2,1}$,
and set
\[
\Lambda=\{\omega\}\cup O_{i_0,1}\cup O_{i_0+1,1}\cup O_{i_0+2,1}.
\]
Let
\[
M=\langle G_\omega,G_{\omega'},G_{\omega''},G_{\omega'''}\rangle.
\]

By the description of the orbits of $G_{\omega'}$ and $G_{\omega''}$ given in~\eqref{casE1} and~\eqref{casE2}, the group $M$ acts transitively on $\Lambda$.  
On the other hand, since $O_{i_0,1}$ is of type~$1$, the subgroup $\langle G_\omega,G_{\omega'}\rangle$ acts transitively on
\[
\Delta=\{\omega\}\cup O_{i_0,1},
\]
and hence $\Delta$ is a block of imprimitivity for the action of $\langle G_\omega,G_{\omega'}\rangle$ on $\Omega$.  
As $\langle G_\omega,G_{\omega'}\rangle\le M$, the set $\Delta$ is also a block of imprimitivity for the action of $M$ on $\Lambda$.
Therefore,
$|\Delta|=1+i_0$  divides
$|\Lambda|=1+i_0+(i_0+1)+(i_0+2)=3i_0+4$, which is a contradiction.  
Hence $m_{i_0+2}=0$, and consequently $m_i=0$ for every $i\ge i_0+2$.

Summarizing, the stabilizer $G_\omega$ has exactly one orbit of size $1$, one orbit of size $i_0$, and all remaining orbits of size $i_0+1$.  
In particular, 
$m_{i_0+1}=n/(i_0+1)-1$.

Let
$\Delta_0 = \{\omega\} \cup O_{i_0,1}$,
$\Delta_j = O_{i_0+1,j}$,
for $1 \le j \le \frac{n}{i_0+1} - 1$.
Set
\[
\Sigma = \{\Delta_0, \Delta_1, \ldots, \Delta_{n/(i_0+1)-1}\},
\quad
N = \langle G_\omega, G_{\omega'} \rangle,\]
where $\omega'\in O_{i_0,1}$.
It follows from~\eqref{casE2} that, if $\omega''$ lies in a $G_\omega$-orbit of type~$2$, then
\[
G_{\omega''}\le \langle G_\omega,G_{\omega'}\rangle = N.
\]
Since all $G_\omega$-orbits of size greater than $i_0$ are of type~$2$, we conclude
$G_\alpha \le N$, for every  $\alpha\in \Omega$.
Thus $N$ contains the stabilizer of every element of $\Omega$, and hence $N$ is a normal subgroup of $G$.  

The $N$-orbits on $\Omega$ are precisely the sets $\Delta_j$ and hence $\Sigma$ is a non-trivial system of imprimitivity for the action of $G$ on $\Omega$, and $N$ is the kernel of the action of $G$ on $\Sigma$.

Since $G$ acts transitively on $\Sigma$ and $N$ is also the setwise stabilizer of the block $\Delta_0$, it follows that $G$ acts regularly on $\Sigma$.  
Moreover, as shown above, the setwise stabilizer of a block acts $2$-transitively on that block.
\end{proof}
\section{Metric dimension of Johnson graphs}\label{section:subsets}
Let $m$ and $k$ be natural numbers with $1\le k\le m-1$, and let $[m]=\{1,\ldots,m\}$.
We consider $G=\mathrm{Sym}(m)$, the symmetric group of degree $m$, in its
natural action on ${[m]\choose k}$, the collection of all $k$-subsets of $[m]$.
Observe that the cardinality of ${[m]\choose k}$ equals ${m\choose k}$.

We investigate the orbit dimension of $G$ in this action, and we denote it by $\sigma(m,k)$ for simplicity.
\subsection{Preliminaries}\label{sec:preliminaries}

Given $\omega\in {[m]\choose k}$, the stabilizer $G_\omega$ has $k+1$ orbits on ${[m]\choose k}$. Indeed, for each $i\in \{0,\ldots,k\}$, the collection of all elements $\omega'\in {[m]\choose k}$ with $|\omega\cap\omega'|=i$ forms a $G_\omega$-orbit. Since this is also the collection of all vertices in the Johnson graph $J(m,k)$ at distance $i$ from $\omega$, we see that $\sigma(m,k)$ coincides with the metric dimension of $J(m,k)$. In graph-theoretic terminology, separating sets are called \textit{\textbf{resolving sets}}.

The case $k=1$ is straightforward. As $G$ is $2$-transitive on ${[m]\choose 1}=[m]$, from Corollary~\ref{cor}, we deduce
$\sigma(m,1)=m-1$.
The case $k=2$ is already more interesting. It was proved by Bailey and Cameron~\cite{bayley,cameron} that
\begin{equation*}
\sigma(m,2)=
\begin{cases}
\lceil 2(m-1)/3\rceil,&\textrm{if }m=5 \text{ or }3 \mid m,\\
\lceil 2(m-1)/3\rceil+1,&\textrm{otherwise.}
\end{cases}
\end{equation*}
Therefore, by~\cite{halasi}, when $k\in \{1,2\}$, $\sigma(m,k)$ either equals the base size of $\mathrm{Sym}(m)$ on ${[m]\choose k}$ or differs from it by at most $1$. We will see that for other values of $k$ this is no longer true.

We say that $S$ is a separating set for ${[m]\choose k}$ to mean that $S\subseteq {[m]\choose k}$ is a separating set for the action of $\mathrm{Sym}(m)$ on ${[m]\choose k}$, or equivalently that $S$ is a resolving set for $J(m,k)$.

We start with some observations.

\begin{lemma}\label{lemma:sepsub}
Let $S$ be a separating set for ${[m]\choose k}$ with $m\ge 2k-1$. Then, for any $i\in\mathbb{N}$ with $i\le k$ and for any two distinct subsets $\omega_1,\omega_2$ of $[m]$ with $|\omega_1|=|\omega_2|=i$, there exists $\alpha\in S$ with $|\omega_1\cap \alpha|\ne |\omega_2\cap \alpha|$.
\end{lemma}

Roughly speaking, Lemma~\ref{lemma:sepsub} says that if $S$ is separating for the $k$-subsets and $m\ge 2k-1$, then $S$ also separates the $i$-subsets for every $i\le k$.

\begin{proof}
Let $i\in\mathbb{N}$ with $i\le k$ and let $\omega_1,\omega_2$ be two distinct $i$-subsets of $[m]$. If $i=k$, then there is nothing to prove, since $S$ is separating for the $k$-subsets of $[m]$. Assume $i<k$. Since
$|[m]\setminus (\omega_1\cup \omega_2)|\ge m-2i$
and $m\ge 2k-1$, we have $m-2i\ge k-i$. Hence there exists a $(k-i)$-subset $\gamma$ of $[m]$ disjoint from $\omega_1\cup \omega_2$. Let $\omega_1'=\omega_1\cup \gamma$ and $\omega_2'=\omega_2\cup \gamma$. Then $\omega_1',\omega_2'$ are two distinct $k$-subsets of $[m]$, and hence there exists $\alpha\in S$ with $|\omega_1'\cap \alpha|\ne |\omega_2'\cap \alpha|$. Since
\[
\omega_1'\cap \alpha=(\omega_1\cap \alpha)\cup (\gamma\cap \alpha), \quad
\omega_2'\cap \alpha=(\omega_2\cap \alpha)\cup (\gamma\cap \alpha),
\]
it follows that $|\omega_1\cap \alpha|\ne |\omega_2\cap \alpha|$.
\end{proof}

\begin{lemma}\label{lemma:plusone}
When $m\ge 2k-1$ and $k\ge 2$, we have $\sigma(m+1,k)\le\sigma(m,k)+\lfloor k/2\rfloor+1$.
\end{lemma}

\begin{proof}

Let $S \subseteq \binom{[m]}{k}$ be a separating family with $|S| = \sigma(m,k)$ and let $\ell=\lfloor k/2\rfloor+1$.
Fix a $k$-subset $\alpha = \{d_1,\dots,d_k\} \in S$. For each $i\in\{1,\ldots,k\}$, define
\[
\beta_i = \{m+1\} \cup (\alpha \setminus \{d_i\}) \in \binom{[m+1]}{k}.
\]
Set $S' = S \cup \{\beta_1,\dots,\beta_\ell\}$.
Then $|S'| = \sigma(m,k) + \ell$. We show that $S'$ is separating for the $k$-subsets of $[m+1]$.

Let $\omega_1,\omega_2 \in \binom{[m+1]}{k}$ be distinct.
If neither contains $m+1$, then both lie in $[m]$ and are separated by $S$. If both contain $m+1$, then $S$ separates $\omega_1\setminus\{m+1\}$ and $\omega_2\setminus\{m+1\}$ by Lemma~\ref{lemma:sepsub}. Hence $S$ separates $\omega_1,\omega_2$.
It remains to consider the case that only one of $\omega_1,\omega_2$ contains $m+1$. Without loss of generality, we may suppose that $m+1\in\omega_1$. Let
$\omega_1 = A \cup \{m+1\}$
with $A \in \binom{[m]}{k-1}$. Assume, arguing by contradiction, that no element in $S$ separates $\omega_1$ and $\omega_2$, and that none of the sets $\beta_1,\dots,\beta_\ell$ separates them. Then $|\omega_1\cap \alpha|=|\omega_2\cap \alpha|$ and,  for each $i\in \{1,\ldots,\ell\}$, 
$|\beta_i \cap (A \cup \{m+1\})|
= |\beta_i \cap \omega_2|$.
As $m+1\in \omega_1$ and $m+1\notin \alpha$, we deduce $|A \cap \alpha|=|\omega_2\cap \alpha|$. Let $x=|A\cap \alpha|=|\omega_2\cap\alpha|$. Define
\[
a_i=
\begin{cases}1&\text{if }d_i\in A,\\
0&\text{otherwise,}
\end{cases}
\qquad b_i=
\begin{cases}1&\text{if }d_i\in \omega_2,\\
0&\text{otherwise.}
\end{cases}
\]
Then
$|\beta_i \cap (A \cup \{m+1\})|
= 1 + x - a_i$ and 
$|\beta_i \cap \omega_2|
= x - b_i$,
so
$1 + x - a_i = x - b_i$,
that is,
$1=a_i-b_i$. This implies $a_i=1$ (that is, $d_i\in A\subseteq\omega_1$) and $b_i=0$ (that is, $d_i\notin\omega_2$). Since this argument holds for each $i\in \{1,\ldots,\ell\}$, we deduce $x=|A\cap \alpha|\ge \ell$ and $x=|\omega_2\cap \alpha|\le k-\ell$. Therefore, $\ell\le k-\ell$,
which is impossible for our choice of $\ell$. Therefore some $\beta_i$ (with $i\in \{1,\ldots,\ell\}$) separates $\omega_1$ and $\omega_2$.
\end{proof}
Let $S$ be a collection of $k$-subsets of $[m]$. We say that $S$ is a \textbf{\textit{robust separating set}} for ${[m]\choose k}$ if
\begin{itemize}
\item for any two distinct $k$-subsets $\omega,\omega'$ of $[m]$, there exists $\alpha\in S$ with $|\omega\cap \alpha|\ne |\omega'\cap \alpha|$, that is, $S$ is separating for the action of $\mathrm{Sym}(m)$ on ${[m]\choose k}$, and
\item for any two distinct subsets $\omega,\omega'$ of $[m]$ with $|\omega|,|\omega'|<k$, there exists $\alpha\in S$ with $|\omega\cap \alpha|\ne |\omega'\cap \alpha|$.
\end{itemize}
We let $\sigma^\ast(m,k)$ be the minimal cardinality of a robust separating set for ${[m]\choose k}$. In particular, $\sigma^\ast(m,k)\ge \sigma(m,k)$, and in general the inequality can be strict. In Table~\ref{table:0}, we report the values of $\sigma(m,3)$ and $\sigma^\ast(m,3)$ for small values of $m$.

\begin{table}[!ht]
\begin{tabular}{r|cccccccccccc}\hline
$m$              &4&5&6&7&8&9&10&11&12&13\\\hline
$\sigma(m,3)$&3&3&4&5&5&6&7  & 7&8&8\\\hline
$\sigma^\ast(m,3)$&3&4&5&5&5&6&7&7&8&9\\
\end{tabular}
\caption{Values of $\sigma(m,3)$ and $\sigma^\ast(m,3)$ for small $m$}\label{table:0}
\end{table}

We find that $\sigma^\ast(m,k)$ behaves more regularly than $\sigma(m,k)$.

\begin{lemma}\label{lemma:robust}
Let $X_1$ and $X_2$ be two disjoint sets with $|X_1|,|X_2|>k$. Let $S_1$ be a robust separating set for the $k$-subsets of $X_1$ and let $S_2$ be a robust separating set for the $k$-subsets of $X_2$. Let $X=X_1\cup X_2$ and $S=S_1\cup S_2$. Then $S$ is a robust separating set for the $k$-subsets of $X$. In particular,
\[
\sigma^\ast(|X_1|+|X_2|,k)\le \sigma^\ast(|X_1|,k)+\sigma^\ast(|X_2|,k).
\]
\end{lemma}

\begin{proof}
Let $\omega_1$ and $\omega_2$ be two distinct subsets of $X$ with either $|\omega_1|,|\omega_2|<k$ or $|\omega_1|=k=|\omega_2|$. Since $\omega_1\ne\omega_2$, there exists $i\in \{1,2\}$ such that $X_i\cap\omega_1\ne X_i\cap\omega_2$.

If either $|X_i\cap \omega_{1}|,|X_i\cap\omega_2|<k$ or $|X_i\cap\omega_1|=k=|X_i\cap\omega_2|$, then there exists $\alpha\in S_i$ with
\[
|(X_i\cap \omega_{1})\cap\alpha|\ne |(X_i\cap \omega_{2})\cap\alpha|,
\]
since $S_i$ is robust. As $\alpha\subseteq X_i$, it follows that $|\omega_1\cap\alpha|\ne |\omega_2\cap\alpha|$, and hence $\alpha$ separates $\omega_1$ and $\omega_2$.

Thus, relabeling if necessary, we may assume that $\omega_1\subseteq X_1$, $\omega_2\subseteq X_2$ and $|\omega_1|=k=|\omega_2|$. Since $|X_1|>k$, we have $S_1\ne\emptyset$, and hence there exists $\alpha\in S_1$ with $\omega_1\cap\alpha\ne\emptyset$, while $\omega_2\cap\alpha=\emptyset$. Therefore $\alpha$ separates $\omega_1$ and $\omega_2$, and $S$ is robust.

The inequality follows by taking $S_i$ with $|S_i|=\sigma^\ast(|X_i|,k)$.
\end{proof}
In the next corollary, we show how Lemma~\ref{lemma:robust} can be applied.
\begin{corollary}\label{cor:1}
Let $m=9q+r$ where $q,r\in\mathbb{N}$ and $0\le r<9$. Then $\sigma(m,3)\le 2q+2r\le \frac{2m}{3}+16$.
\end{corollary}

\begin{proof}
From Table~\ref{table:0}, $\sigma^\ast(9,3)=6$. For instance,
\[
\{\{1,2,3\},\{1,2,4\},\{1,2,5\},\{1,6,8\},\{1,6,9\},\{2,6,7\}\}
\]
is a robust separating set for ${[9]\choose 3}$. Hence, by Lemma~\ref{lemma:robust}, we have
$\sigma^\ast(9q,3)\le q\,\sigma^\ast(9,3)=6q$.
Thus $\sigma(9q,3)\le \sigma^\ast(9q,3)\le 6q$. Finally, Lemma~\ref{lemma:plusone} gives
$\sigma(m,3)\le \sigma(9q,3)+2r\le 6q+2r\le 6q+16$.
\end{proof}

The argument in Corollary~\ref{cor:1} appears deceptively simple, but the upper bound obtained is already asymptotically better than the strongest upper bound in~\cite{baileyyy} (namely $\lfloor 3(m+1)/4\rfloor$). The difficulty in applying Lemma~\ref{lemma:robust} for general values of $k$ lies in constructing small robust separating sets. We overcome this difficulty in the next section.

\subsection{An adjacency matrix and some notation}\label{sec:adja}

Let $S$ be a subset of ${[m]\choose k}$ of cardinality $\ell$.
We consider the $m\times\ell$ matrix $M_S$, whose rows are indexed by the elements of $[m]$ and whose columns are indexed by the elements of $S$. We set $(M_S)_{i,s}=1$ if $i\in s$ and $(M_S)_{i,s}=0$ otherwise. Since every element $s\in S$ has cardinality $k$, each column of $M_S$ has exactly $k$ ones.

Given $i\in [m]$, we let $(M_S)_i$ denote the $i$th row of $M_S$, and we say that $(M_S)_i$ has \textit{\textbf{weight}} $x$ if it has exactly $x$ ones, that is, if $i$ belongs to exactly $x$ elements of $S$. The \textit{\textbf{support}} of $(M_S)_i$ is the collection of all $s\in S$ with $(M_S)_{i,s}=1$, that is, the set of elements of $S$ containing $i$. More generally, given rows $(M_S)_{i_1},\ldots,(M_S)_{i_t}$, we define their support as the union of their supports.

The fact that $S$ is separating can be detected using the matrix $M_S$.

\begin{lemma}\label{lemma:sepsub1}
When $m\ge 2k-1$, the set $S$ is separating for ${[m]\choose k}$ if and only if, for every $i\in \{0,\ldots,k\}$ and for any two distinct $i$-subsets $x$ and $y$ of $[m]$, we have
\[
\sum_{j\in x}(M_S)_j \ne \sum_{j\in y}(M_S)_j.
\]
\end{lemma}

\begin{proof}
By the definition of $M_S$, the set $S$ is separating if and only if for any two distinct $k$-subsets $x$ and $y$ of $[m]$ we have
$\sum_{j\in x}(M_S)_j \ne \sum_{j\in y}(M_S)_j$.
The general statement then follows from Lemma~\ref{lemma:sepsub}.
\end{proof}

For each $x\in \{0,\ldots,\ell\}$, let $D_x$ be the set of all elements $i\in [m]$ such that $(M_S)_i$ has weight $x$. Then $D_0,\ldots,D_\ell$ form a partition of $[m]$, and we let $d_x=|D_x|$. Thus
\begin{equation}\label{bestio1}
m=\sum_{x=0}^\ell d_x.
\end{equation}

\begin{example}\label{exa:1}{\rm For instance, $$S=\{
\{2, 11, 13\},
\{1, 6, 10\},
\{1, 5, 8\},
\{4, 5, 13\},
\{7, 10, 12\},
\{3, 12, 13\},
\{2, 9, 12\},
\{5, 6, 7\}
\}$$
is a separating set for ${[13]\choose 3}$. Here $D_0=\emptyset$, $D_1=\{3,4,8,9,11\}$, $D_2=\{1,2,6,7,10\}$, $D_3=\{5,12,13\}$ and $D_i=\emptyset$ for every $i\ge 4$.
}
\end{example}
Now consider the set $\{(i,s)\in [m]\times S\mid i\in s\}$. By double counting its elements, we obtain
\begin{equation}\label{bestio2}
k\ell=\sum_{x=0}^\ell x d_x.
\end{equation}

For the rest of this section, we assume that $S$ is separating for ${[m]\choose k}$.
We claim that
\begin{equation}\label{bestio4}
d_0\le 1.
\end{equation}
Assume, by contradiction, that $d_0\ge 2$. Then there exist two distinct elements $i,j\in [m]$ such that $(M_S)_i$ and $(M_S)_j$ have weight zero, and hence $(M_S)_i=(M_S)_j$. This contradicts Lemma~\ref{lemma:sepsub1} (applied with $x=\{i\}$ and $y=\{j\}$), since $S$ is separating.

We now claim that
\begin{align}\label{bestio5}
\forall i,j\in D_1,\ \text{ if }(M_S)_i=(M_S)_j,\ \text{ then } i=j.
\end{align}
Recall that the elements in $D_1$ are those of weight $1$. Suppose, for a contradiction, that there exist two distinct elements $i,j\in D_1$ with $(M_S)_i=(M_S)_j$. Again, this contradicts Lemma~\ref{lemma:sepsub1} applied with $x=\{i\}$ and $y=\{j\}$. Hence~\eqref{bestio5} holds, and in particular
$d_1\le \ell$.

Now write $S=\{s_1,\ldots,s_\ell\}$ and define
\[
A=\{s\in S\mid \text{there exists } j\in D_1 \text{ with } j\in s\}.
\]
The set $A$ admits a natural interpretation in terms of the matrix $M_S$: it consists of the supports of the elements of $D_1$. By~\eqref{bestio5}, no two elements of $D_1$ have the same support. We also let $B=S\setminus A$.

Next, using the elements in $D_2$, we construct a graph $\Gamma$ with vertex set $S$. Observe that, if $i\in D_2$, then $(M_S)_i$ has weight $2$, that is, $i$ is contained in exactly two elements of $S$. We declare $\{s_x,s_y\}$ to be an edge of $\Gamma$ if there exists $i\in D_2$ such that $(M_S)_i$ has support $\{s_x,s_y\}$. Now, let 
\begin{itemize}
\item $e(A,A)$ be the number of edges of $\Gamma$ within $A$,
\item $e(B,B)$  be the number of edges of $\Gamma$ within $B$, and we let 
\item $e(A,B)$ be the number of edges of $\Gamma$ having one end in $A$ and one end in $B$.
\end{itemize}

\begin{example}{\rm In Example~\ref{exa:1}, if we label the elements in $S$ with  $s_1=\{2, 11, 13\}$,
$s_2=\{1, 6, 10\}$,
$s_3=\{1, 5, 8\}$,
$s_4=\{4, 5, 13\}$,
$s_5=\{7, 10, 12\}$,
$s_6=\{3, 12, 13\}$,
$s_7=\{2, 9, 12\}$,
$s_8=\{5, 6, 7\}$, then $A=\{s_1,s_3,s_4,s_6,s_7\}$ and $B=\{s_2,s_5,s_8\}$. The graph $\Gamma$ for this example is depicted in Figure~\ref{fig:ex1}. Thus $e(A,A)=1$, $e(A,B)=1$ and $e(B,B)=3$.
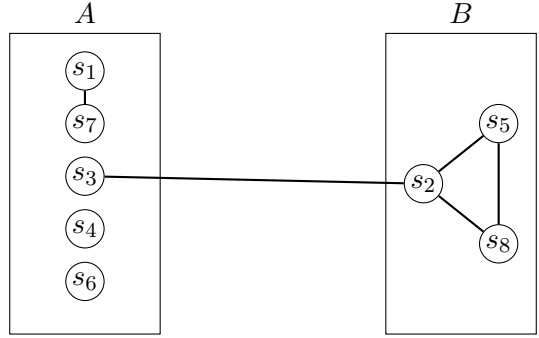
\begin{figure}[!ht]
\begin{tikzpicture}[
    node_style/.style={circle, draw, minimum size=2mm, inner sep=1pt},
    box/.style={draw, rectangle, minimum width=2cm, minimum height=4cm}
]

\node[box] (Abox) at (-1,0) {};
\node[above] at (Abox.north) {$A$};

\node[node_style] (s1) at (-1,1.5) {$s_1$};
\node[node_style] (s7) at (-1.0,0.8) {$s_7$};
\node[node_style] (s3) at (-1.0,0.1) {$s_3$};
\node[node_style] at (-1.0,-0.6) {$s_4$};
\node[node_style] at (-1.0,-1.3) {$s_6$};
\draw[thick] (s1)--(s7);
\node[box] (Bbox) at (4,0) {};
\node[above] at (Bbox.north) {$B$};

\node[node_style](s5) at (4.5,0.8) {$s_5$};
\node[node_style](s2) at (3.5,0.0) {$s_2$};
\node[node_style] (s8) at (4.5,-0.8) {$s_8$};
\draw[thick](s2)--(s3);
\draw[thick](s2)--(s5);
\draw[thick](s2)--(s8);
\draw[thick](s5)--(s8);
\end{tikzpicture}
\caption{The auxiliary graph $\Gamma$ for the separating set in Example~\ref{exa:1}.}\label{fig:ex1}
\end{figure}
}

\end{example}

Observe that $\Gamma$ is a simple graph, that is, it has no multiple edges. Indeed, there are no two distinct elements $i,j\in [m]$ with $(M_S)_i=(M_S)_j$, since otherwise Lemma~\ref{lemma:sepsub1} would contradict the fact that $S$ is separating. Therefore
$d_2$ equals the number of edges of $\Gamma$, that is, $d_2=e(A,A)+e(A,B)+e(B,B)$.

We claim that
\begin{equation}\label{bestio7}
\text{each vertex of }\Gamma \text{ has valency at most } k.
\end{equation}
Indeed, if $s\in S$ has valency greater than $k$, then $s$ contains more than $k$ elements of $D_2$, which is impossible since every element of $S$ has cardinality $k$.

We claim that
\begin{equation}\label{bestio8}
A \text{ contains at most one edge of }\Gamma, \text{ that is, }e(A,A)\le1.
\end{equation}
Suppose, for a contradiction, that $A$ contains at least two edges $\{s_x,s_y\}$ and $\{s_{x'},s_{y'}\}$ of $\Gamma$. Since $s_x,s_y,s_{x'},s_{y'}\in A$, there exist elements $i_x,i_y,i_{x'},i_{y'}\in D_1$ such that $(M_S)_{i_x}$ has support $s_x$, $(M_S)_{i_y}$ has support $s_y$, $(M_S)_{i_{x'}}$ has support $s_{x'}$, and $(M_S)_{i_{y'}}$ has support $s_{y'}$. Moreover, there exist $i,i'\in D_2$ such that $(M_S)_i$ has support $\{s_x,s_y\}$ and $(M_S)_{i'}$ has support $\{s_{x'},s_{y'}\}$.
Then
\[
(M_S)_i+(M_S)_{i_{x'}}+(M_S)_{i_{y'}}=(M_S)_{i'}+(M_S)_{i_x}+(M_S)_{i_y},
\]
which contradicts Lemma~\ref{lemma:sepsub1}, since $S$ is separating.

We claim that
\begin{equation}\label{bestio9}
\forall u\in B,\ \text{there exists at most one }a\in A \text{ with }\{a,u\} \text{ an edge of }\Gamma.
\end{equation}
Suppose, for a contradiction, that there exists $u\in B$ and two distinct elements $a,a'\in A$ such that both $\{a,u\}$ and $\{a',u\}$ are edges of $\Gamma$. By the definition of $D_1$, there exist $i,i'\in [m]$ such that $(M_S)_{i}$ has support $a$ and $(M_S)_{i'}$ has support $a'$. Similarly, by the definition of $D_2$, there exist $j,j'\in [m]$ such that $(M_S)_{j}$ has support $\{a,u\}$ and $(M_S)_{j'}$ has support $\{a',u\}$. Then
\[
(M_S)_{j}+(M_S)_{i'}=(M_S)_{i}+(M_S)_{j'},
\]
which contradicts Lemma~\ref{lemma:sepsub1}, since $S$ is separating.
Now, since $|B|=\ell-d_1$, and each vertex in $A$ has at most $k-1$ neighbours in $B$ by~\eqref{bestio7}, it follows from~\eqref{bestio9} that
\begin{equation}\label{bestio9bis}
e(A,B)\le \min(\ell-d_1,(k-1)d_1).
\end{equation}

Next, consider the set $\{(s,e)\mid s\in B,\ e \text{ an edge of }\Gamma,\ s\in e\}$. By double counting its elements and using~\eqref{bestio7}, we obtain
$e(A,B)+2e(B,B)\le k|B|=k(\ell-d_1)$.
Thus, by~\eqref{bestio8} and~\eqref{bestio9bis}, we have
\[
2d_2=2e(A,A)+2e(A,B)+2e(B,B)\le 2+(k+1)(\ell-d_1).
\]
Hence
\begin{equation}\label{bestio10}
d_2\le 1+\frac{(k+1)}{2}(\ell-d_1).
\end{equation}

\begin{proof}[Proof of Theorem~$\ref{thrm:main}$]We start by proving the lower bound on $\sigma(m,k)$. Using~\eqref{bestio1},~\eqref{bestio2},~\eqref{bestio4} and~\eqref{bestio10}, we obtain
\begin{align}\label{eq:mondaymorning}
k\ell
&\ge d_1+2d_2+3\sum_{x=3}^\ell d_x=d_1+2d_2+3\left(\sum_{x=0}^\ell d_x-d_0-d_1-d_2\right)\\\nonumber
&=d_1+2d_2+3(m-d_0-d_1-d_2)=3m-(3d_0+2d_1+d_2)\\\nonumber
&\ge 3m-4-\frac{k+1}{2}(\ell-d_1)-2d_1=3m-4-\frac{k+1}{2}\ell+\frac{k-3}{2}d_1.\nonumber
\end{align}
Since $k\ge 3$, we have $\frac{k-3}{2}d_1\ge 0$, and hence
$\frac{3k+1}{2}\,\ell \ge 3m-4$.
This yields the desired lower bound.

\smallskip

We now turn to the upper bound. Set
$k'=4\big((k-1)^{2k}-1\big)/(k-2)$.
Let $\Delta$ be a $k$-regular graph of girth at least $2k+1$ with the minimum possible number of vertices. Recall that a graph is $k$-regular if every vertex has valency $k$, and that the girth is the length of a shortest cycle. By a theorem of Erd\H{o}s and Sachs~\cite{ES}, the number of vertices of $\Delta$ is at most $k'$.
Let $X_\Delta$ be the set of edges of $\Delta$, and let $S_\Delta$ be the set of vertices. Let $M_{S_\Delta}$ be the incidence matrix of $\Delta$, with rows indexed by edges and columns indexed by vertices, defined by
\[
(M_{S_\Delta})_{e,v}=1 \quad\text{if and only if}\quad v\in e.
\]
Then each row contains exactly $2$ ones and each column contains exactly $k$ ones.

Using this matrix, each column (i.e.\ each vertex) determines a $k$-subset of $X_\Delta$, namely the set of edges incident with that vertex. In this way we obtain a family of $|S_\Delta|$ $k$-subsets of $X_\Delta$. With a slight abuse of notation, we identify this family with $S_\Delta$ itself.

We claim that $S_\Delta$ is a robust separating set. Suppose not. Then, by Lemma~\ref{lemma:sepsub1}, there exist two distinct collections $\{e_1,\ldots,e_x\}$ and $\{e'_1,\ldots,e'_{x'}\}$ of edges of $\Delta$, with either $x,x'<k$ or $x=k=x'$, such that
\begin{equation}\label{equation:girth}
\sum_{i=1}^x (M_{S_\Delta})_{e_i}
=
\sum_{i=1}^{x'} (M_{S_\Delta})_{e'_i}.
\end{equation}
Here each $(M_{S_\Delta})_{e_i}$ is a row vector indexed by vertices.

Now consider the subgraph $\Delta'$ of $\Delta$ formed by the edges $e_1,\ldots,e_x,e'_1,\ldots,e'_{x'}$. Since $\Delta$ has no cycles of length $\le 2k$, and $\Delta'$ has at most $x+x'\le 2k$ edges, it follows that $\Delta'$ contains no cycles, and hence is a forest. Therefore $\Delta'$ has a vertex $v$ of valency $1$. Without loss of generality, suppose that $v$ is incident to an edge in $\{e_1,\ldots,e_x\}$. Then the $v$-coordinate of the left-hand side of~\eqref{equation:girth} is $1$, while the $v$-coordinate of the right-hand side is $0$, since $v$ is not incident to any edge among $e'_1,\ldots,e'_{x'}$. This is a contradiction.
Thus $S_\Delta$ is a robust separating set. It follows that
\[
\frac{\sigma^\ast(k',k)}{k'}
\le
\frac{|S_\Delta|}{k'}
=
\frac{|S_\Delta|}{|X_\Delta|}
=
\frac{2}{k},
\]
where the last equality follows from the fact that $\Delta$ is $k$-regular, so $k|S_\Delta|=2|X_\Delta|$. The proof now follows as in Corollary ~\ref{cor:1}.
\end{proof}

\subsection{The case $k=3$: proof of the lower bound in Theorem~\ref{thrm:main2}}\label{sec:k=3}

In this section we use the notation of Section~\ref{sec:adja} and refine the bounds established there, specifically for $k=3$.

To obtain this refinement, we let $B_1$ be the set of elements $b\in B$ such that there exists an edge $\{b,a\}$ of $\Gamma$ with $a\in A$, and we let $B_0=B\setminus B_1$. With this notation,~\eqref{bestio9} becomes
\begin{equation}\label{bestio11}
|B_1|=e(A,B).
\end{equation}

Now, the vertices in $B_0$ have neighbours only in $B$, and hence they have valency at most $3$. Similarly, the vertices in $B_1$ have at most $2$ neighbours in $B$, since each of them has one neighbour in $A$. Thus, by double counting the degrees of the vertices in the subgraph of $\Gamma$ induced on $B$, we obtain
$2e(B,B)\le 3|B_0|+2|B_1|$.
Therefore,
\begin{align*}
d_2
&=e(A,A)+e(A,B)+e(B,B)\le 1+|B_1|+\frac{3}{2}|B_0|+|B_1|=1+2(|B_0|+|B_1|)-\frac{|B_0|}{2}\\
&=1+2|B|-\frac{|B_0|}{2}
=1+2\ell-2d_1-\frac{|B_0|}{2}.
\end{align*}
Using~\eqref{bestio4}, we deduce
\[
3d_0+2d_1+d_2\le 4+2\ell-\frac{|B_0|}{2}.
\]
Arguing as in~\eqref{eq:mondaymorning}, we obtain
\begin{equation}\label{bestio11bis}
\ell\ge \frac{3m-4}{5}+\frac{|B_0|}{10}.
\end{equation}

As above, we let 
\begin{itemize}
\item $e(B_0,B_0)$ be the number of edges of $\Gamma$ within $B_0$,
\item $e(B_1,B_1)$ be the number of edges of $\Gamma$ within $B_1$,
\item $e(B_0,B_1)$ be the number of edges of $\Gamma$ having one end in $B_0$ and one end in $B_1$.
\end{itemize}
Thus $e(B,B)=e(B_0,B_0)+e(B_0,B_1)+e(B_1,B_1)$.

We claim that
\begin{equation}\label{eq:monday1}
\text{the subgraph of }\Gamma\text{ induced on }B\text{ has no paths of length }2\text{ whose endvertices lie in }B_1.
\end{equation}
Indeed, suppose, for a contradiction, that there exists a path of length $2$, say $b_1,b_2,b_3$ in $B$ with $b_1,b_3\in B_1$. By definition of $B_1$, there exist $a_1,a_3\in A$ such that $\{a_1,b_1\}$ and $\{a_3,b_3\}$ are edges of $\Gamma$. In particular, there exist $x_1,x_2,x_3,x_4\in D_2$ such that $(M_S)_{x_1}$ has support $\{a_1,b_1\}$, $(M_S)_{x_2}$ has support $\{b_1,b_2\}$, $(M_S)_{x_3}$ has support $\{b_2,b_3\}$, and $(M_S)_{x_4}$ has support $\{a_3,b_3\}$. Moreover, since $a_1,a_3\in A$, there exist $x_0,x_5\in D_1$ such that $(M_S)_{x_0}$ has support $a_1$ and $(M_S)_{x_5}$ has support $a_3$. See Figure~\ref{fig1}.
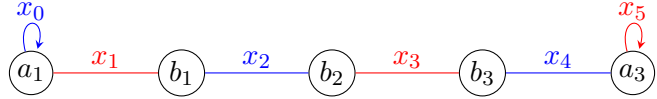
\begin{figure}[!ht]
\begin{tikzpicture}[>=stealth, every node/.style={inner sep=1.5pt}]
  \node[circle, draw] (a1) at (0,0) {$a_1$};
  \node[circle, draw] (b1) at (2,0) {$b_1$};
  \node[circle, draw] (b2) at (4,0) {$b_2$};
  \node[circle, draw] (b3) at (6,0) {$b_3$};
  \node[circle, draw] (a3) at (8,0) {$a_3$};

  \draw[red] (a1) -- node[above] {$x_1$} (b1);
  \draw[blue] (b1) -- node[above] {$x_2$} (b2);
  \draw[red] (b2) -- node[above] {$x_3$} (b3);
  \draw[blue] (b3) -- node[above] {$x_4$} (a3);

  \draw[blue] (a1) edge[loop above] node {$x_0$} ();
  \draw[red] (a3) edge[loop above] node {$x_5$} ();
\end{tikzpicture}
\caption{Auxiliary figure for the proof of~\eqref{eq:monday1}}\label{fig1}
\end{figure}
Therefore,
\[
(M_S)_{x_0}+(M_S)_{x_2}+(M_S)_{x_4}=(M_S)_{x_1}+(M_S)_{x_3}+(M_S)_{x_5},
\]
which contradicts Lemma~\ref{lemma:sepsub1}.

Now,~\eqref{eq:monday1} implies that the subgraph of $\Gamma$ induced on $B_1$ is a disjoint union of edges and isolated vertices. In particular,
\begin{equation}\label{bestio12}
e(B_1,B_1)\le \frac{|B_1|}{2}.
\end{equation}
Similarly,~\eqref{eq:monday1} implies that every element of $B_0$ has at most one neighbour in $B_1$, and hence
\begin{equation}\label{bestio13}
e(B_0,B_1)\le |B_0|.
\end{equation}

Observe that, in the subgraph induced by $\Gamma$ on $B_0$, there are at most $|B_0|-e(B_0,B_1)$ vertices of degree $3$, since each vertex of $B_0$ having a neighbour in $B_1$ can have at most $2$ neighbours within $B_0$. Therefore,
\begin{equation}\label{monday2}
2e(B_0,B_0)\le 3(|B_0|-e(B_0,B_1))+2e(B_0,B_1)=3|B_0|-e(B_0,B_1).
\end{equation}

Thus, from~\eqref{bestio12},~\eqref{bestio13} and~\eqref{monday2}, we deduce
\begin{align}\label{bestio23}
e(B,B)
&=e(B_0,B_0)+e(B_0,B_1)+e(B_1,B_1)\le \frac{3}{2}|B_0|+\frac{1}{2}e(B_0,B_1)+\frac{|B_1|}{2}\\
&\le \frac{3}{2}|B_0|+\frac{|B_0|}{2}+\frac{|B_1|}{2}=2|B_0|+\frac{|B_1|}{2}.\nonumber
\end{align}

Hence, from this bound and from~\eqref{bestio8} and~\eqref{bestio11}, we obtain
\begin{align*}
d_2
&=e(A,A)+e(A,B)+e(B,B)\le 1+|B_1|+2|B_0|+\frac{|B_1|}{2}\\
&=1+2|B|-\frac{|B_1|}{2}
=1+2\ell-2d_1-\frac{|B_1|}{2}.
\end{align*}
Using~\eqref{bestio4}, we deduce
\[
3d_0+2d_1+d_2\le 4+2\ell-\frac{|B_1|}{2},
\]
and hence, arguing as in~\eqref{eq:mondaymorning}, we obtain
\begin{equation}\label{bestio11ter}
\ell\ge \frac{3m-4}{5}+\frac{|B_1|}{10}.
\end{equation}

Combining~\eqref{bestio11bis} and~\eqref{bestio11ter}, we obtain
\begin{equation}\label{bestio4fold}
\ell\ge\frac{3m-4}{5}+\frac{1}{10}\max(|B_0|,|B_1|)
\ge \frac{3m-4}{5}+\frac{|B|}{20}
=\frac{3m-4}{5}+\frac{\ell-d_1}{20}.
\end{equation}

Next, we consider the elements in $D_3$. Recall that the support of an element of $D_3$ has cardinality $3$. Let $\{s_1,s_2,s_3\}$ be the support of $x\in D_3$. We say that $x$ has type $AAA$ if $\{s_1,s_2,s_3\}\subseteq A$, has type $AAB$ if $|\{s_1,s_2,s_3\}\cap A|=2$, has type $ABB$ if $|\{s_1,s_2,s_3\}\cap A|=1$, and has type $BBB$ if $\{s_1,s_2,s_3\}\subseteq B$. We denote by $e(A,A,A)$, $e(A,A,B)$, $e(A,B,B)$ and $e(B,B,B)$ the number of elements of each type. Thus $d_3=e(A,A,A)+e(A,A,B)+e(A,B,B)+e(B,B,B)$.

We claim that 
\begin{equation}\label{bestio16+}
\text{the supports of any two distinct elements of type }AAA\text{ are disjoint}.
\end{equation}
Suppose, for a contradiction, that $\{s_1,s_2,s_3\}$ and $\{s_4,s_5,s_6\}$ are the supports of two distinct elements $x,y\in D_3$ of type $AAA$, and that their intersection is nonempty.
If $|\{s_1,s_2,s_3\}\cap \{s_4,s_5,s_6\}|=2$, then without loss of generality we may assume that $s_1=s_4$ and $s_2=s_5$. Since $s_3,s_6\in A$, there exist $y_3,y_6\in D_1$ such that $(M_S)_{y_3}$ has support $s_3$ and $(M_S)_{y_6}$ has support $s_6$. Moreover, $(M_S)_x$ and $(M_S)_y$ have supports $\{s_1,s_2,s_3\}$ and $\{s_1,s_2,s_6\}$ respectively. Hence
\[
(M_S)_x+(M_S)_{y_6}=(M_S)_y+(M_S)_{y_3},
\]
which contradicts Lemma~\ref{lemma:sepsub1}. The case $|\{s_1,s_2,s_3\}\cap \{s_4,s_5,s_6\}|=1$ is similar. This proves~\eqref{bestio16+}.

Since $d_1=|A|$, we deduce
\begin{equation}\label{bestio16}
e(A,A,A)\le\frac{d_1}{3}.
\end{equation}

Let $b\in B$. We claim that 
\begin{equation}\label{bestio17+}
\text{there exists at most one }x\in D_3\text{ of type $AAB$ such that the support of }(M_S)_x\text{ contains }b.
\end{equation}
Suppose, for a contradiction, that there exist two distinct elements $x,x'\in D_3$ of type $AAB$ such that $b$ lies in both supports. Let $\{a_1,a_2,b\}$ and $\{a_3,a_4,b\}$ be the supports of $(M_S)_x$ and $(M_S)_{x'}$, respectively. Since $x,x'$ are of type $AAB$, we have $a_1,a_2,a_3,a_4\in A$. Hence there exist $y_1,y_2,y_3,y_4\in D_1$ such that $(M_S)_{y_i}$ has support $a_i$ for $i=1,2,3,4$. Therefore,
\[
(M_S)_x+(M_S)_{y_3}+(M_S)_{y_4}
=(M_S)_{x'}+(M_S)_{y_1}+(M_S)_{y_2},
\]
which contradicts Lemma~\ref{lemma:sepsub1}. This proves~\eqref{bestio17+}. Consequently,
\begin{equation}\label{bestio17}
e(A,A,B)\le |B|.
\end{equation}

Recall that each column of $M_S$ contains exactly $3$ ones. Counting the contributions coming from the elements of $B$, we obtain
\[
e(A,B)+2e(B,B)+e(A,A,B)+2e(A,B,B)+3e(B,B,B)\le 3|B|.
\]
Therefore,
\begin{equation}\label{bestio19}
e(A,A,B)+2e(A,B,B)+3e(B,B,B)\le 3|B|-e(A,B)-2e(B,B),
\end{equation}
and hence
\begin{equation}\label{bestio20}
e(A,A,B)+2e(A,B,B)\le 3|B|-e(A,B)-2e(B,B).
\end{equation}

Combining~\eqref{bestio19} and~\eqref{bestio20}, we obtain
\[
\frac{3}{2}e(A,A,B)+3e(A,B,B)+3e(B,B,B)
\le \frac{9}{2}|B|-\frac{3}{2}e(A,B)-3e(B,B).
\]
Using~\eqref{bestio17}, we deduce
\[
3e(A,A,B)+3e(A,B,B)+3e(B,B,B)
\le 6|B|-\frac{3}{2}e(A,B)-3e(B,B),
\]
and hence, from~\eqref{bestio16},
\begin{align*}
d_3
&=e(A,A,A)+e(A,A,B)+e(A,B,B)+e(B,B,B)\\
&\le \frac{d_1}{3}+2|B|-\frac{1}{2}e(A,B)-e(B,B).
\end{align*}

Since $d_2=e(A,A)+e(A,B)+e(B,B)$, from~\eqref{bestio8},~\eqref{bestio11} and~\eqref{bestio23}, we obtain
\begin{align*}
2d_2+d_3
&\le 2+\frac{d_1}{3}+2|B|+\frac{3}{2}e(A,B)+e(B,B)\\
&\le 2+\frac{d_1}{3}+2|B|+\frac{3}{2}|B_1|+2|B_0|+\frac{|B_1|}{2}\\
&=2+\frac{d_1}{3}+4|B|.
\end{align*}

Thus, using~\eqref{bestio4}, we obtain
$4d_0+3d_1+2d_2+d_3\le 6+4\ell-2d_1/3$.
From~\eqref{eq:mondaymorning}, we have
\[
3\ell\ge 4m-(4d_0+3d_1+2d_2+d_3)\ge 4m-6-4\ell+\frac{2}{3}d_1,
\]
and hence
\begin{equation}\label{bestio24}
\ell\ge\frac{4m-6}{7}+\frac{2d_1}{21}.
\end{equation}

Interpolating~\eqref{bestio4fold} and~\eqref{bestio24}, after some computations, we obtain
$\ell\ge \frac{36m-50}{59}$,
which proves the lower bound in Theorem~\ref{thrm:main2}.

\subsection{The case $k=3$: proof of the upper bound in Theorem~\ref{thrm:main2}}\label{sec:k=3upper}

Let $m$ be a positive integer and write $m=13q+r$ with $q,r\in\mathbb{N}$ and $0\le r<13$. To prove the upper bound in Theorem~\ref{thrm:main2}, it suffices to treat the case $r=0$, since the general case follows by repeated applications of Lemma~\ref{lemma:plusone}. Thus we assume that $m=13q$. 

For $q\in \{1,2\}$, the existence of a separating set of cardinality $8q$ can be verified by computer. Hence, for the remainder of the argument we assume that $q\ge 3$.

We construct a separating set $S$ for ${[m]\choose 3}$ by specifying its adjacency matrix $M_S$ (see Section~\ref{sec:adja}). In our construction, $M_S$ has no row of weight zero, and all rows have weight at most $3$. Therefore, $M_S$ can be encoded by a hypergraph $\mathcal{H}$ whose vertices correspond to the columns of $M_S$ (that is, the elements of $S$), and whose hyperedges correspond to the rows of $M_S$. More precisely, using the notation from Section~\ref{sec:adja}:
\begin{itemize}
\item hyperedges of size $1$ correspond to rows of weight $1$ (elements of $D_1$),
\item hyperedges of size $2$ correspond to rows of weight $2$ (elements of $D_2$),
\item hyperedges of size $3$ correspond to rows of weight $3$ (elements of $D_3$).
\end{itemize}
Thus, we describe the hypergraph $\mathcal{H}$ instead of the matrix $M_S$.

In our construction, the hypergraph $\mathcal{H}$ is represented in Figures~\ref{figA},~\ref{figB} and~\ref{figC}, depending on the congruence class of $q \pmod{3}$. For instance, when $q\equiv 2\pmod{3}$, one should refer to Figure~\ref{figC}. For illustration, the case $q=3$ is shown in Figure~\ref{figD}.

The hypergraph $\mathcal{H}$ has $8q$ vertices, represented either by solid green nodes or by black circles. There are $6q$ solid green vertices and $2q$ black vertices. For example, when $q=3$, $\mathcal{H}$ has $24$ vertices, of which $18$ are solid green and $6$ are black, as shown in Figure~\ref{figD}.

\begin{figure}[!ht]
\begin{tikzpicture}[scale=1.1,rotate=90] \node[circle,draw=green!60!black, fill=green!60!black,inner sep=1.6pt] (L1) at (-.5,0) {}; \node[circle,draw=black,inner sep=1.6pt] (R1) at (.5,0){}; \node[circle,draw=green!60!black, fill=green!60!black,inner sep=1.6pt] (L2) at (-.5,-1) {}; \node[circle,draw=black,inner sep=1.6pt] (R2) at (.5,-1){}; \draw[thick,blue] (L1) -- (R1); \draw[thick,blue] (L2)--(R2); \draw[thick,blue] (R1)--(R2); \node[circle,draw=green!60!black, fill=green!60!black,inner sep=1.6pt] (L9) at (-.5,-2) {}; \node[circle,draw=black,inner sep=1.6pt] (R9) at (.5,-2){}; \node[circle,draw=green!60!black, fill=green!60!black,inner sep=1.6pt] (L10) at (-.5,-3) {}; \node[circle,draw=black,inner sep=1.6pt] (R10) at (.5,-3){}; \node[circle,draw=black,inner sep=1.6pt] (R11) at (1.5,-3){}; \draw[thick,blue] (L9)--(R9); \draw[thick,blue] (L10)--(R10); \draw[thick,blue] (R9)--(R10); \draw[thick,blue] (R10)--(R11); \node[circle,draw=green!60!black, fill=green!60!black,inner sep=1.6pt] (L11) at (-.5,-4) {}; \node[circle,draw=black,inner sep=1.6pt] (R12) at (.5,-4) {}; \draw[thick,blue] (L11)--(R12);
\node[circle,draw=green!60!black, fill=green!60!black,inner sep=1.6pt] (L12) at (-.5,-5) {}; \node[circle,draw=green!60!black, fill=green!60!black,inner sep=1.6pt] (L13) at (-.5,-6) {}; \node[circle,draw=green!60!black, fill=green!60!black,inner sep=1.6pt] (L14) at (.5,-5) {}; \draw[thick,blue] (L12)--(L13); \node[circle,draw=green!60!black, fill=green!60!black,inner sep=1.6pt] (RR1) at (1.5,0) {}; \node[circle,draw=green!60!black, fill=green!60!black,inner sep=1.6pt] (RR1b) at (2.5,0) {}; \node[circle,draw=green!60!black, fill=green!60!black,inner sep=1.6pt] (RR2) at (1.5,-1) {}; \node[circle,draw=green!60!black, fill=green!60!black,inner sep=1.6pt] (RR2b) at (2.5,-1) {}; \node[circle,draw=green!60!black, fill=green!60!black,inner sep=1.6pt] (RR9) at (1.5,-2) {}; \node[circle,draw=green!60!black, fill=green!60!black,inner sep=1.6pt] (RR9b) at (2.5,-2) {}; \node[circle,draw=green!60!black, fill=green!60!black,inner sep=1.6pt] (RR10) at (2.5,-3) {}; \node[circle,draw=green!60!black, fill=green!60!black,inner sep=1.6pt] (RR10b) at (3.5,-3) {}; \node[circle,draw=green!60!black, fill=green!60!black,inner sep=1.6pt] (RR11) at (1.5,-4) {}; \node[circle,draw=green!60!black, fill=green!60!black,inner sep=1.6pt] (RR11b) at (2.5,-4) {};
\draw[thick,decorate,decoration={snake,amplitude=0.1mm, segment length=1mm}] (R1)--(RR1)--(RR1b); \draw[thick,decorate,decoration={snake,amplitude=0.1mm, segment length=1mm}] (R2)--(RR2)--(RR2b); \draw[thick,decorate,decoration={snake,amplitude=0.1mm, segment length=1mm}] (R9)--(RR9)--(RR9b); \draw[thick,decorate,decoration={snake,amplitude=0.1mm, segment length=1mm}] (R11)--(RR10)--(RR10b); \draw[thick,decorate,decoration={snake,amplitude=0.1mm, segment length=1mm}] (R12)--(RR11)--(RR11b); \draw[thick,red,decorate,decoration={snake,amplitude=0.1mm, segment length=1mm}] (L1)--(L2)--(L9); \draw[thick,red,decorate,decoration={snake,amplitude=0.1mm, segment length=1mm}] (L10)--(L11)--(L12); \draw[thick,red,decorate,decoration={snake,amplitude=0.1mm, segment length=1mm}] (L13)to[out=130, in=150,looseness=1.4](RR1)--(RR2); \draw[thick,red,decorate,decoration={snake,amplitude=0.1mm, segment length=1mm}] (RR1b)--(RR2b)--(RR9); \draw[thick,red,decorate,decoration={snake,amplitude=0.1mm, segment length=1mm}] (RR9b)--(RR10)--(RR11); \draw[thick,red,decorate,decoration={snake,amplitude=0.1mm, segment length=1mm}] (L14)--(RR11b)--(RR10b); \draw[thick,color=yellow!70!black,decorate,decoration={snake,amplitude=0.1mm, segment length=1mm}] (L14) --(R12)--(R11); \end{tikzpicture} \caption{The hypergraph when $q=3$}\label{figD}
\end{figure}

We now describe the hyperedges of $\mathcal{H}$.

\begin{itemize}
\item $\mathcal{H}$ has $6q$ hyperedges of cardinality $1$, corresponding to the elements of $D_1$. In the figures, these are represented by the solid green vertices. In particular, this set coincides with the set $A$ introduced in Section~\ref{sec:adja}, while the black vertices correspond to the set $B$.

\item $\mathcal{H}$ has $3q$ hyperedges of cardinality $2$, represented by solid blue edges. These edges form a structured pattern consisting of two disjoint edges, $q-2$ paths of length $3$, and one distinguished path of length $3$ with an additional edge attached at one of its internal vertices. 

In terms of the notation of Sections~\ref{sec:adja} and~\ref{sec:k=3}, we have  $e(A,A)=1$, $e(A,B)=2q-1$ and $e(B,B)=q$.
To clarify this structure, we have drawn only the blue edges of $\mathcal{H}$ in Figure~\ref{figDD}. Note that the placement of the two disjoint edges depends on the drawing: in Figure~\ref{figA} they both appear on the right-hand side (with one of them lying inside a hyperedge of cardinality $3$), whereas in Figure~\ref{figB} they appear on opposite sides of the diagram.
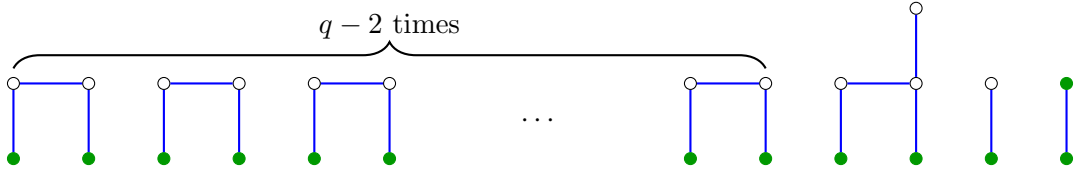
\begin{figure}[!ht]
\begin{tikzpicture}
\node[circle,draw=black,inner sep=1.6pt] (R1) at (-10,.5){};
\node[circle,draw=green!60!black, fill=green!60!black,inner sep=1.6pt] (S1) at (-10,-.5){};
\node[circle,draw=black,inner sep=1.6pt] (R2) at (-9,.5){};
\node[circle,draw=green!60!black, fill=green!60!black,inner sep=1.6pt] (S2) at (-9,-.5){};
\draw[thick,blue] (R1)--(S1);\draw[thick,blue] (R2)--(S2);\draw[thick,blue] (R1)--(R2);

\node[circle,draw=black,inner sep=1.6pt] (R3) at (-8,.5){};
\node[circle,draw=green!60!black, fill=green!60!black,inner sep=1.6pt] (S3) at (-8,-.5){};
\node[circle,draw=black,inner sep=1.6pt] (R4) at (-7,.5){};
\node[circle,draw=green!60!black, fill=green!60!black,inner sep=1.6pt] (S4) at (-7,-.5){};
\draw[thick,blue] (R3)--(S3);\draw[thick,blue] (R4)--(S4);\draw[thick,blue] (R3)--(R4);

\node[circle,draw=black,inner sep=1.6pt] (R5) at (-6,.5){};
\node[circle,draw=green!60!black, fill=green!60!black,inner sep=1.6pt] (S5) at (-6,-.5){};
\node[circle,draw=black,inner sep=1.6pt] (R6) at (-5,.5){};
\node[circle,draw=green!60!black, fill=green!60!black,inner sep=1.6pt] (S6) at (-5,-.5){};
\draw[thick,blue] (R5)--(S5);\draw[thick,blue] (R6)--(S6);\draw[thick,blue] (R5)--(R6);

\node at (-3,0) {$\cdots$};

\node[circle,draw=black,inner sep=1.6pt] (R7) at (-1,.5){};
\node[circle,draw=green!60!black, fill=green!60!black,inner sep=1.6pt] (S7) at (-1,-.5){};
\node[circle,draw=black,inner sep=1.6pt] (R8) at (0,.5){};
\node[circle,draw=green!60!black, fill=green!60!black,inner sep=1.6pt] (S8) at (0,-.5){};
\draw[thick,blue] (R7)--(S7);\draw[thick,blue] (R8)--(S8);\draw[thick,blue] (R7)--(R8);

\node[circle,draw=black,inner sep=1.6pt] (R9) at (1,.5){};
\node[circle,draw=green!60!black, fill=green!60!black,inner sep=1.6pt] (S9) at (1,-.5){};
\node[circle,draw=black,inner sep=1.6pt] (R10) at (2,.5){};
\node[circle,draw=black,inner sep=1.6pt] (R11) at (2,1.5){};
\node[circle,draw=green!60!black, fill=green!60!black,inner sep=1.6pt] (S10) at (2,-.5){};
\draw[thick,blue] (R9)--(S9);\draw[thick,blue] (R10)--(S10);\draw[thick,blue] (R9)--(R10);\draw[thick,blue](R11)--(R10);

\node[circle,draw=black,inner sep=1.6pt] (R12) at (3,.5){};
\node[circle,draw=green!60!black, fill=green!60!black,inner sep=1.6pt] (S12) at (3,-.5){};

\node[circle,draw=green!60!black, fill=green!60!black,inner sep=1.6pt] (R13) at (4,.5){};
\node[circle,draw=green!60!black, fill=green!60!black,inner sep=1.6pt] (S13) at (4,-.5){};
\draw[thick,blue](R12)--(S12);\draw[thick,blue](R13)--(S13);
\draw [thick, decorate,decoration={brace,amplitude=10pt}] (-10,.7) -- (0,.7) node [black,midway,yshift=16pt] {$q-2$ times};
\end{tikzpicture}
\caption{The underlying graph for the hyperedges of cardinality $2$ of $\mathcal{H}$}\label{figDD}
\end{figure}

\item $\mathcal{H}$ has $4q$ hyperedges of cardinality $3$, represented by coloured (red, black, or yellow) curves in the figures. These are distributed as follows:
\begin{itemize}
\item exactly one hyperedge of type $ABB$ (the yellow one), so $e(A,B,B)=1$,
\item $2q-1$ hyperedges of type $AAB$ (the black ones), so $e(A,A,B)=2q-1$,
\item the remaining $2q$ hyperedges are of type $AAA$ (the red ones), and they partition the set $A$. Thus $e(A,A,A)=2q$.
\end{itemize}
\end{itemize}

Summarizing, $\mathcal{H}$ has
$6q+3q+4q=13q$
hyperedges, and hence $M_S$ has $13q$ rows and $8q$ columns. Moreover, as one can easily verify from Figures~\ref{figA},~\ref{figB} and~\ref{figC}, each column of $M_S$ contains exactly three ones, since every vertex of $\mathcal{H}$ is incident with exactly three hyperedges. Therefore, $S$ is indeed a collection of $3$-subsets.

It remains to prove that $S$ is separating. This follows from a case-by-case analysis. We outline the main ideas and omit routine verifications.
Suppose, for a contradiction, that there exist two distinct $3$-subsets of indices $\{i_1,i_2,i_3\}$ and $\{j_1,j_2,j_3\}$ such that
\begin{equation}\label{last}
(M_S)_{i_1}+(M_S)_{i_2}+(M_S)_{i_3}
=
(M_S)_{j_1}+(M_S)_{j_2}+(M_S)_{j_3}.
\end{equation}

First assume that all indices lie in $D_1\cup D_2$, that is, no row involved has weight $3$. Then~\eqref{last} corresponds to a configuration in $\mathcal{H}$ involving only hyperedges of size $1$ and $2$. One checks that this would force one of the following configurations:
\begin{itemize}
\item a cycle of length $6$ consisting of blue edges,
\item a $4$-cycle of blue edges together with a disjoint blue edge,
\item a path of length $4$ of blue edges whose endvertices lie in $A$,
\item at least two distinct blue edges joining vertices in $A$.
\end{itemize}
A direct inspection of $\mathcal{H}$ shows that none of these configurations occurs, and hence~\eqref{last} is impossible in this case.

The case where~\eqref{last} involves rows of weight $3$ is treated similarly, by analyzing the interaction between hyperedges of size $3$ and those of size $1$ and $2$. 
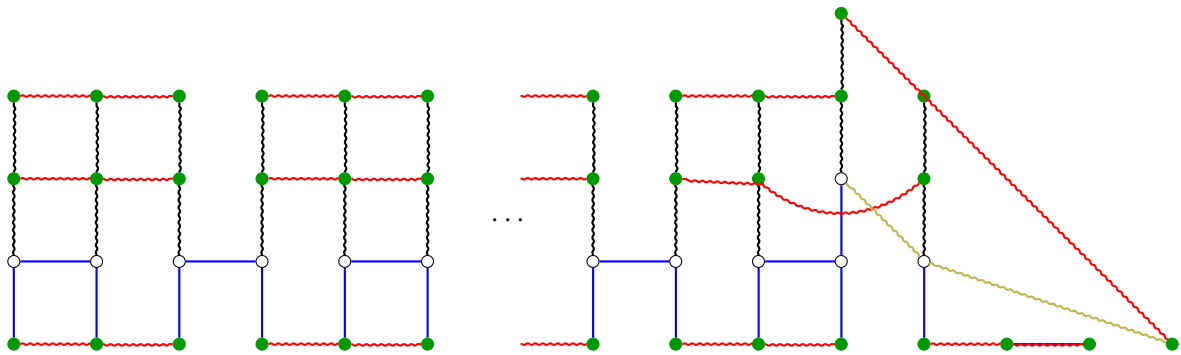
\begin{figure}[!ht]
\begin{tikzpicture}[scale=1.1,rotate=90]
\node[circle,draw=green!60!black, fill=green!60!black,inner sep=1.6pt] (L1) at (-.5,0) {};
\node[circle,draw=black,inner sep=1.6pt] (R1) at (.5,0){};
\node[circle,draw=green!60!black, fill=green!60!black,inner sep=1.6pt] (L2) at (-.5,-1) {};
\node[circle,draw=black,inner sep=1.6pt] (R2) at (.5,-1){};
\node[circle,draw=green!60!black, fill=green!60!black,inner sep=1.6pt] (L3) at (-.5,-2) {};
\node[circle,draw=black,inner sep=1.6pt] (R3) at (.5,-2){};
\node[circle,draw=green!60!black, fill=green!60!black,inner sep=1.6pt] (L4) at (-.5,-3) {};
\node[circle,draw=black,inner sep=1.6pt] (R4) at (.5,-3){};
\node[circle,draw=green!60!black, fill=green!60!black,inner sep=1.6pt] (L5) at (-.5,-4) {};
\node[circle,draw=black,inner sep=1.6pt] (R5) at (.5,-4){};
\node[circle,draw=green!60!black, fill=green!60!black,inner sep=1.6pt] (L6) at (-.5,-5) {};
\node[circle,draw=black,inner sep=1.6pt] (R6) at (.5,-5){};
\draw[thick,blue] (L1) -- (R1);
\draw[thick,blue] (L2)--(R2);
\draw[thick,blue] (R1)--(R2);
\draw[thick,blue] (L3) -- (R3);
\draw[thick,blue] (L4)--(R4);
\draw[thick,blue] (R3)--(R4);
\draw[thick,blue] (L5) -- (R5);
\draw[thick,blue] (L6)--(R6);
\draw[thick,blue] (R5)--(R6);

\node at (1,-6) {$\cdots$};

\node[circle,draw=green!60!black, fill=green!60!black,inner sep=1.6pt] (L7) at (-.5,-7) {};
\node[circle,draw=black,inner sep=1.6pt] (R7) at (.5,-7){};
\node[circle,draw=green!60!black, fill=green!60!black,inner sep=1.6pt] (L8) at (-.5,-8) {};
\node[circle,draw=black,inner sep=1.6pt] (R8) at (.5,-8){};
\draw[thick,blue] (L7)--(R7);
\draw[thick,blue] (L8)--(R8);
\draw[thick,blue] (R7)--(R8);

\node[circle,draw=green!60!black, fill=green!60!black,inner sep=1.6pt] (L9) at (-.5,-9) {};
\node[circle,draw=black,inner sep=1.6pt] (R9) at (.5,-9){};
\node[circle,draw=green!60!black, fill=green!60!black,inner sep=1.6pt] (L10) at (-.5,-10) {};
\node[circle,draw=black,inner sep=1.6pt] (R10) at (.5,-10){};
\node[circle,draw=black,inner sep=1.6pt] (R11) at (1.5,-10){};
\draw[thick,blue] (L9)--(R9);
\draw[thick,blue] (L10)--(R10);
\draw[thick,blue] (R9)--(R10);
\draw[thick,blue] (R10)--(R11);

\node[circle,draw=green!60!black, fill=green!60!black,inner sep=1.6pt] (L11) at (-.5,-11) {};
\node[circle,draw=black,inner sep=1.6pt] (R12) at (.5,-11) {};
\draw[thick,blue] (L11)--(R12);

\node[circle,draw=green!60!black, fill=green!60!black,inner sep=1.6pt] (L12) at (-.5,-12) {};
\node[circle,draw=green!60!black, fill=green!60!black,inner sep=1.6pt] (L13) at (-.5,-13) {};
\node[circle,draw=green!60!black, fill=green!60!black,inner sep=1.6pt] (L14) at (-.5,-14) {};
\draw[thick,blue] (L12)--(L13);

\node[circle,draw=green!60!black, fill=green!60!black,inner sep=1.6pt] (RR1) at (1.5,0) {};
\node[circle,draw=green!60!black, fill=green!60!black,inner sep=1.6pt] (RR1b) at (2.5,0) {};
\node[circle,draw=green!60!black, fill=green!60!black,inner sep=1.6pt] (RR2) at (1.5,-1) {};
\node[circle,draw=green!60!black, fill=green!60!black,inner sep=1.6pt] (RR2b) at (2.5,-1) {};
\node[circle,draw=green!60!black, fill=green!60!black,inner sep=1.6pt] (RR3) at (1.5,-2) {};
\node[circle,draw=green!60!black, fill=green!60!black,inner sep=1.6pt] (RR3b) at (2.5,-2) {};
\node[circle,draw=green!60!black, fill=green!60!black,inner sep=1.6pt] (RR4) at (1.5,-3) {};
\node[circle,draw=green!60!black, fill=green!60!black,inner sep=1.6pt] (RR4b) at (2.5,-3) {};
\node[circle,draw=green!60!black, fill=green!60!black,inner sep=1.6pt] (RR5) at (1.5,-4) {};
\node[circle,draw=green!60!black, fill=green!60!black,inner sep=1.6pt] (RR5b) at (2.5,-4) {};
\node[circle,draw=green!60!black, fill=green!60!black,inner sep=1.6pt] (RR6) at (1.5,-5) {};
\node[circle,draw=green!60!black, fill=green!60!black,inner sep=1.6pt] (RR6b) at (2.5,-5) {};

\node[circle,draw=green!60!black, fill=green!60!black,inner sep=1.6pt] (RR7) at (1.5,-7) {};
\node[circle,draw=green!60!black, fill=green!60!black,inner sep=1.6pt] (RR7b) at (2.5,-7) {};
\node[circle,draw=green!60!black, fill=green!60!black,inner sep=1.6pt] (RR8) at (1.5,-8) {};
\node[circle,draw=green!60!black, fill=green!60!black,inner sep=1.6pt] (RR8b) at (2.5,-8) {};
\node[circle,draw=green!60!black, fill=green!60!black,inner sep=1.6pt] (RR9) at (1.5,-9) {};
\node[circle,draw=green!60!black, fill=green!60!black,inner sep=1.6pt] (RR9b) at (2.5,-9) {};

\node[circle,draw=green!60!black, fill=green!60!black,inner sep=1.6pt] (RR10) at (2.5,-10) {};
\node[circle,draw=green!60!black, fill=green!60!black,inner sep=1.6pt] (RR10b) at (3.5,-10) {};
\node[circle,draw=green!60!black, fill=green!60!black,inner sep=1.6pt] (RR11) at (1.5,-11) {};
\node[circle,draw=green!60!black, fill=green!60!black,inner sep=1.6pt] (RR11b) at (2.5,-11) {};
\draw[thick,red,decorate,decoration={snake,amplitude=0.1mm, segment length=1mm}] (L1)--(L2)--(L3);
\draw[thick,red,decorate,decoration={snake,amplitude=0.1mm, segment length=1mm}] (L4)--(L5)--(L6);
\draw[thick,red,decorate,decoration={snake,amplitude=0.1mm, segment length=1mm}] (L11)--(L12)--(L13);
\draw[thick,red,decorate,decoration={snake,amplitude=0.1mm, segment length=1mm}] (L8)--(L9)--(L10);
\node (L7-) at (-.5,-6){};
\node (L7--) at (1.5,-6){};
\node (L7---) at (2.5,-6){};
\draw[thick,red,decorate,decoration={snake,amplitude=0.1mm, segment length=1mm}] (L7)--(L7-);
\draw[thick,decorate,decoration={snake,amplitude=0.1mm, segment length=1mm}] (R1)--(RR1)--(RR1b);
\draw[thick,decorate,decoration={snake,amplitude=0.1mm, segment length=1mm}] (R2)--(RR2)--(RR2b);
\draw[thick,decorate,decoration={snake,amplitude=0.1mm, segment length=1mm}] (R3)--(RR3)--(RR3b);
\draw[thick,decorate,decoration={snake,amplitude=0.1mm, segment length=1mm}] (R4)--(RR4)--(RR4b);
\draw[thick,decorate,decoration={snake,amplitude=0.1mm, segment length=1mm}] (R5)--(RR5)--(RR5b);
\draw[thick,decorate,decoration={snake,amplitude=0.1mm, segment length=1mm}] (R6)--(RR6)--(RR6b);
\draw[thick,decorate,decoration={snake,amplitude=0.1mm, segment length=1mm}] (R7)--(RR7)--(RR7b);
\draw[thick,decorate,decoration={snake,amplitude=0.1mm, segment length=1mm}] (R8)--(RR8)--(RR8b);
\draw[thick,decorate,decoration={snake,amplitude=0.1mm, segment length=1mm}] (R9)--(RR9)--(RR9b);
\draw[thick,decorate,decoration={snake,amplitude=0.1mm, segment length=1mm}] (R11)--(RR10)--(RR10b);
\draw[thick,decorate,decoration={snake,amplitude=0.1mm, segment length=1mm}] (R12)--(RR11)--(RR11b);

\draw[thick,red,decorate,decoration={snake,amplitude=0.1mm, segment length=1mm}] (RR1)--(RR2)--(RR3);
\draw[thick,red,decorate,decoration={snake,amplitude=0.1mm, segment length=1mm}] (RR1b)--(RR2b)--(RR3b);
\draw[thick,red,decorate,decoration={snake,amplitude=0.1mm, segment length=1mm}] (RR4)--(RR5)--(RR6);
\draw[thick,red,decorate,decoration={snake,amplitude=0.1mm, segment length=1mm}] (RR4b)--(RR5b)--(RR6b);

\draw[thick,red,decorate,decoration={snake,amplitude=0.1mm, segment length=1mm}] (L14)--(RR10b);
\draw[thick,red,decorate,decoration={snake,amplitude=0.1mm, segment length=1mm}] (RR11) to[bend left=40](RR9)--(RR8);
\draw[thick,red,decorate,decoration={snake,amplitude=0.1mm, segment length=1mm}] (RR8b) --(RR9b)--(RR10);
\draw[thick,red,decorate,decoration={snake,amplitude=0.1mm, segment length=1mm}] (RR7) --(L7--);
\draw[thick,red,decorate,decoration={snake,amplitude=0.1mm, segment length=1mm}] (RR7b) --(L7---);
\draw[thick,color=yellow!70!black,decorate,decoration={snake,amplitude=0.1mm, segment length=1mm}] (L14) --(R12)--(R11);
\end{tikzpicture}
\caption{The hypergraph when $q\equiv 1\pmod 3$}\label{figA}
\end{figure}

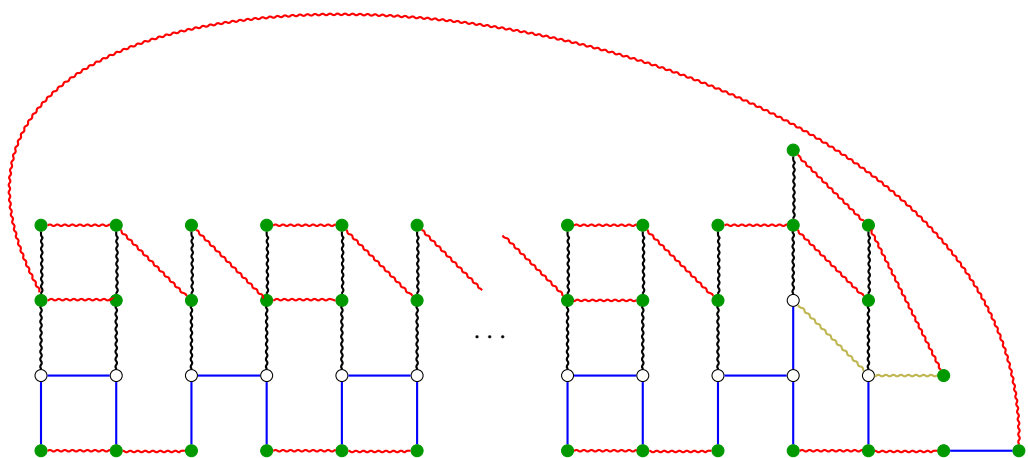
\begin{figure}[!ht]
\begin{tikzpicture}[scale=1,rotate=90]
\node[circle,draw=green!60!black, fill=green!60!black,inner sep=1.6pt] (L1) at (-.5,0) {};
\node[circle,draw=black,inner sep=1.6pt] (R1) at (.5,0){};
\node[circle,draw=green!60!black, fill=green!60!black,inner sep=1.6pt] (L2) at (-.5,-1) {};
\node[circle,draw=black,inner sep=1.6pt] (R2) at (.5,-1){};
\node[circle,draw=green!60!black, fill=green!60!black,inner sep=1.6pt] (L3) at (-.5,-2) {};
\node[circle,draw=black,inner sep=1.6pt] (R3) at (.5,-2){};
\node[circle,draw=green!60!black, fill=green!60!black,inner sep=1.6pt] (L4) at (-.5,-3) {};
\node[circle,draw=black,inner sep=1.6pt] (R4) at (.5,-3){};
\node[circle,draw=green!60!black, fill=green!60!black,inner sep=1.6pt] (L5) at (-.5,-4) {};
\node[circle,draw=black,inner sep=1.6pt] (R5) at (.5,-4){};
\node[circle,draw=green!60!black, fill=green!60!black,inner sep=1.6pt] (L6) at (-.5,-5) {};
\node[circle,draw=black,inner sep=1.6pt] (R6) at (.5,-5){};
\draw[thick,blue] (L1) -- (R1);
\draw[thick,blue] (L2)--(R2);
\draw[thick,blue] (R1)--(R2);
\draw[thick,blue] (L3) -- (R3);
\draw[thick,blue] (L4)--(R4);
\draw[thick,blue] (R3)--(R4);
\draw[thick,blue] (L5) -- (R5);
\draw[thick,blue] (L6)--(R6);
\draw[thick,blue] (R5)--(R6);

\node at (1,-6) {$\cdots$};

\node[circle,draw=green!60!black, fill=green!60!black,inner sep=1.6pt] (L7) at (-.5,-7) {};
\node[circle,draw=black,inner sep=1.6pt] (R7) at (.5,-7){};
\node[circle,draw=green!60!black, fill=green!60!black,inner sep=1.6pt] (L8) at (-.5,-8) {};
\node[circle,draw=black,inner sep=1.6pt] (R8) at (.5,-8){};
\draw[thick,blue] (L7)--(R7);
\draw[thick,blue] (L8)--(R8);
\draw[thick,blue] (R7)--(R8);

\node[circle,draw=green!60!black, fill=green!60!black,inner sep=1.6pt] (L9) at (-.5,-9) {};
\node[circle,draw=black,inner sep=1.6pt] (R9) at (.5,-9){};
\node[circle,draw=green!60!black, fill=green!60!black,inner sep=1.6pt] (L10) at (-.5,-10) {};
\node[circle,draw=black,inner sep=1.6pt] (R10) at (.5,-10){};
\node[circle,draw=black,inner sep=1.6pt] (R11) at (1.5,-10){};
\draw[thick,blue] (L9)--(R9);
\draw[thick,blue] (L10)--(R10);
\draw[thick,blue] (R9)--(R10);
\draw[thick,blue] (R10)--(R11);

\node[circle,draw=green!60!black, fill=green!60!black,inner sep=1.6pt] (L11) at (-.5,-11) {};
\node[circle,draw=black,inner sep=1.6pt] (R12) at (.5,-11) {};
\draw[thick,blue] (L11)--(R12);

\node[circle,draw=green!60!black, fill=green!60!black,inner sep=1.6pt] (L12) at (-.5,-12) {};
\node[circle,draw=green!60!black, fill=green!60!black,inner sep=1.6pt] (L13) at (-.5,-13) {};
\node[circle,draw=green!60!black, fill=green!60!black,inner sep=1.6pt] (L14) at (.5,-12) {};
\draw[thick,blue] (L12)--(L13);

\node[circle,draw=green!60!black, fill=green!60!black,inner sep=1.6pt] (RR1) at (1.5,0) {};
\node[circle,draw=green!60!black, fill=green!60!black,inner sep=1.6pt] (RR1b) at (2.5,0) {};
\node[circle,draw=green!60!black, fill=green!60!black,inner sep=1.6pt] (RR2) at (1.5,-1) {};
\node[circle,draw=green!60!black, fill=green!60!black,inner sep=1.6pt] (RR2b) at (2.5,-1) {};
\node[circle,draw=green!60!black, fill=green!60!black,inner sep=1.6pt] (RR3) at (1.5,-2) {};
\node[circle,draw=green!60!black, fill=green!60!black,inner sep=1.6pt] (RR3b) at (2.5,-2) {};
\node[circle,draw=green!60!black, fill=green!60!black,inner sep=1.6pt] (RR4) at (1.5,-3) {};
\node[circle,draw=green!60!black, fill=green!60!black,inner sep=1.6pt] (RR4b) at (2.5,-3) {};
\node[circle,draw=green!60!black, fill=green!60!black,inner sep=1.6pt] (RR5) at (1.5,-4) {};
\node[circle,draw=green!60!black, fill=green!60!black,inner sep=1.6pt] (RR5b) at (2.5,-4) {};
\node[circle,draw=green!60!black, fill=green!60!black,inner sep=1.6pt] (RR6) at (1.5,-5) {};
\node[circle,draw=green!60!black, fill=green!60!black,inner sep=1.6pt] (RR6b) at (2.5,-5) {};

\node[circle,draw=green!60!black, fill=green!60!black,inner sep=1.6pt] (RR7) at (1.5,-7) {};
\node[circle,draw=green!60!black, fill=green!60!black,inner sep=1.6pt] (RR7b) at (2.5,-7) {};
\node[circle,draw=green!60!black, fill=green!60!black,inner sep=1.6pt] (RR8) at (1.5,-8) {};
\node[circle,draw=green!60!black, fill=green!60!black,inner sep=1.6pt] (RR8b) at (2.5,-8) {};
\node[circle,draw=green!60!black, fill=green!60!black,inner sep=1.6pt] (RR9) at (1.5,-9) {};
\node[circle,draw=green!60!black, fill=green!60!black,inner sep=1.6pt] (RR9b) at (2.5,-9) {};

\node[circle,draw=green!60!black, fill=green!60!black,inner sep=1.6pt] (RR10) at (2.5,-10) {};
\node[circle,draw=green!60!black, fill=green!60!black,inner sep=1.6pt] (RR10b) at (3.5,-10) {};
\node[circle,draw=green!60!black, fill=green!60!black,inner sep=1.6pt] (RR11) at (1.5,-11) {};
\node[circle,draw=green!60!black, fill=green!60!black,inner sep=1.6pt] (RR11b) at (2.5,-11) {};
\draw[thick,red,decorate,decoration={snake,amplitude=0.1mm, segment length=1mm}] (L1)--(L2)--(L3);
\draw[thick,red,decorate,decoration={snake,amplitude=0.1mm, segment length=1mm}] (L4)--(L5)--(L6);
\draw[thick,red,decorate,decoration={snake,amplitude=0.1mm, segment length=1mm}] (L10)--(L11)--(L12);
\draw[thick,red,decorate,decoration={snake,amplitude=0.1mm, segment length=1mm}] (L7)--(L8)--(L9);
\node (L7-) at (-.5,-6){};
\node (L7--) at (1.5,-6){};
\node (L7---) at (2.5,-6){};
\draw[thick,decorate,decoration={snake,amplitude=0.1mm, segment length=1mm}] (R1)--(RR1)--(RR1b);
\draw[thick,decorate,decoration={snake,amplitude=0.1mm, segment length=1mm}] (R2)--(RR2)--(RR2b);
\draw[thick,decorate,decoration={snake,amplitude=0.1mm, segment length=1mm}] (R3)--(RR3)--(RR3b);
\draw[thick,decorate,decoration={snake,amplitude=0.1mm, segment length=1mm}] (R4)--(RR4)--(RR4b);
\draw[thick,decorate,decoration={snake,amplitude=0.1mm, segment length=1mm}] (R5)--(RR5)--(RR5b);
\draw[thick,decorate,decoration={snake,amplitude=0.1mm, segment length=1mm}] (R6)--(RR6)--(RR6b);
\draw[thick,decorate,decoration={snake,amplitude=0.1mm, segment length=1mm}] (R7)--(RR7)--(RR7b);
\draw[thick,decorate,decoration={snake,amplitude=0.1mm, segment length=1mm}] (R8)--(RR8)--(RR8b);
\draw[thick,decorate,decoration={snake,amplitude=0.1mm, segment length=1mm}] (R9)--(RR9)--(RR9b);
\draw[thick,decorate,decoration={snake,amplitude=0.1mm, segment length=1mm}] (R11)--(RR10)--(RR10b);
\draw[thick,decorate,decoration={snake,amplitude=0.1mm, segment length=1mm}] (R12)--(RR11)--(RR11b);

\draw[thick,red,decorate,decoration={snake,amplitude=0.1mm, segment length=1mm}] (L13)to[out=0, in=30,looseness=1.3](RR1)--(RR2);
\draw[thick,red,decorate,decoration={snake,amplitude=0.1mm, segment length=1mm}] (RR1b)--(RR2b)--(RR3);
\draw[thick,red,decorate,decoration={snake,amplitude=0.1mm, segment length=1mm}] (RR3b)--(RR4)--(RR5);
\draw[thick,red,decorate,decoration={snake,amplitude=0.1mm, segment length=1mm}] (RR4b)--(RR5b)--(RR6);
\draw[thick,red,decorate,decoration={snake,amplitude=0.1mm, segment length=1mm}] (RR6b) --(L7--);

\draw[thick,red,decorate,decoration={snake,amplitude=0.1mm, segment length=1mm}] (RR8)--(RR7)--(L7---);
\draw[thick,red,decorate,decoration={snake,amplitude=0.1mm, segment length=1mm}] (RR7b)--(RR8b)--(RR9);
\draw[thick,red,decorate,decoration={snake,amplitude=0.1mm, segment length=1mm}] (RR9b) --(RR10)--(RR11);
\draw[thick,red,decorate,decoration={snake,amplitude=0.1mm, segment length=1mm}] (L14) --(RR11b)--(RR10b);
\draw[thick,color=yellow!70!black,decorate,decoration={snake,amplitude=0.1mm, segment length=1mm}] (L14) --(R12)--(R11);
\end{tikzpicture}
\caption{The hypergraph when $q\equiv 0\pmod 3$}\label{figB}
\end{figure}

\begin{figure}[!ht]
\begin{tikzpicture}[scale=1.1,rotate=90]
\node[circle,draw=green!60!black, fill=green!60!black,inner sep=1.6pt] (L1) at (-.5,0) {};
\node[circle,draw=black,inner sep=1.6pt] (R1) at (.5,0){};
\node[circle,draw=green!60!black, fill=green!60!black,inner sep=1.6pt] (L2) at (-.5,-1) {};
\node[circle,draw=black,inner sep=1.6pt] (R2) at (.5,-1){};
\node[circle,draw=green!60!black, fill=green!60!black,inner sep=1.6pt] (L3) at (-.5,-2) {};
\node[circle,draw=black,inner sep=1.6pt] (R3) at (.5,-2){};
\node[circle,draw=green!60!black, fill=green!60!black,inner sep=1.6pt] (L4) at (-.5,-3) {};
\node[circle,draw=black,inner sep=1.6pt] (R4) at (.5,-3){};
\node[circle,draw=green!60!black, fill=green!60!black,inner sep=1.6pt] (L5) at (-.5,-4) {};
\node[circle,draw=black,inner sep=1.6pt] (R5) at (.5,-4){};
\node[circle,draw=green!60!black, fill=green!60!black,inner sep=1.6pt] (L6) at (-.5,-5) {};
\node[circle,draw=black,inner sep=1.6pt] (R6) at (.5,-5){};
\draw[thick,blue] (L1) -- (R1);
\draw[thick,blue] (L2)--(R2);
\draw[thick,blue] (R1)--(R2);
\draw[thick,blue] (L3) -- (R3);
\draw[thick,blue] (L4)--(R4);
\draw[thick,blue] (R3)--(R4);
\draw[thick,blue] (L5) -- (R5);
\draw[thick,blue] (L6)--(R6);
\draw[thick,blue] (R5)--(R6);

\node at (1,-6) {$\cdots$};

\node[circle,draw=green!60!black, fill=green!60!black,inner sep=1.6pt] (L7) at (-.5,-7) {};
\node[circle,draw=black,inner sep=1.6pt] (R7) at (.5,-7){};
\node[circle,draw=green!60!black, fill=green!60!black,inner sep=1.6pt] (L8) at (-.5,-8) {};
\node[circle,draw=black,inner sep=1.6pt] (R8) at (.5,-8){};
\draw[thick,blue] (L7)--(R7);
\draw[thick,blue] (L8)--(R8);
\draw[thick,blue] (R7)--(R8);

\node[circle,draw=green!60!black, fill=green!60!black,inner sep=1.6pt] (L9) at (-.5,-9) {};
\node[circle,draw=black,inner sep=1.6pt] (R9) at (.5,-9){};
\node[circle,draw=green!60!black, fill=green!60!black,inner sep=1.6pt] (L10) at (-.5,-10) {};
\node[circle,draw=black,inner sep=1.6pt] (R10) at (.5,-10){};
\node[circle,draw=black,inner sep=1.6pt] (R11) at (1.5,-10){};
\draw[thick,blue] (L9)--(R9);
\draw[thick,blue] (L10)--(R10);
\draw[thick,blue] (R9)--(R10);
\draw[thick,blue] (R10)--(R11);

\node[circle,draw=green!60!black, fill=green!60!black,inner sep=1.6pt] (L11) at (-.5,-11) {};
\node[circle,draw=black,inner sep=1.6pt] (R12) at (.5,-11) {};
\draw[thick,blue] (L11)--(R12);

\node[circle,draw=green!60!black, fill=green!60!black,inner sep=1.6pt] (L12) at (1.5,1) {};
\node[circle,draw=green!60!black, fill=green!60!black,inner sep=1.6pt] (L13) at (2.5,1) {};
\node[circle,draw=green!60!black, fill=green!60!black,inner sep=1.6pt] (L14) at (-.5,-14) {};
\draw[thick,blue] (L12)--(L13);

\node[circle,draw=green!60!black, fill=green!60!black,inner sep=1.6pt] (RR1) at (1.5,0) {};
\node[circle,draw=green!60!black, fill=green!60!black,inner sep=1.6pt] (RR1b) at (2.5,0) {};
\node[circle,draw=green!60!black, fill=green!60!black,inner sep=1.6pt] (RR2) at (1.5,-1) {};
\node[circle,draw=green!60!black, fill=green!60!black,inner sep=1.6pt] (RR2b) at (2.5,-1) {};
\node[circle,draw=green!60!black, fill=green!60!black,inner sep=1.6pt] (RR3) at (1.5,-2) {};
\node[circle,draw=green!60!black, fill=green!60!black,inner sep=1.6pt] (RR3b) at (2.5,-2) {};
\node[circle,draw=green!60!black, fill=green!60!black,inner sep=1.6pt] (RR4) at (1.5,-3) {};
\node[circle,draw=green!60!black, fill=green!60!black,inner sep=1.6pt] (RR4b) at (2.5,-3) {};
\node[circle,draw=green!60!black, fill=green!60!black,inner sep=1.6pt] (RR5) at (1.5,-4) {};
\node[circle,draw=green!60!black, fill=green!60!black,inner sep=1.6pt] (RR5b) at (2.5,-4) {};
\node[circle,draw=green!60!black, fill=green!60!black,inner sep=1.6pt] (RR6) at (1.5,-5) {};
\node[circle,draw=green!60!black, fill=green!60!black,inner sep=1.6pt] (RR6b) at (2.5,-5) {};

\node[circle,draw=green!60!black, fill=green!60!black,inner sep=1.6pt] (RR7) at (1.5,-7) {};
\node[circle,draw=green!60!black, fill=green!60!black,inner sep=1.6pt] (RR7b) at (2.5,-7) {};
\node[circle,draw=green!60!black, fill=green!60!black,inner sep=1.6pt] (RR8) at (1.5,-8) {};
\node[circle,draw=green!60!black, fill=green!60!black,inner sep=1.6pt] (RR8b) at (2.5,-8) {};
\node[circle,draw=green!60!black, fill=green!60!black,inner sep=1.6pt] (RR9) at (1.5,-9) {};
\node[circle,draw=green!60!black, fill=green!60!black,inner sep=1.6pt] (RR9b) at (2.5,-9) {};

\node[circle,draw=green!60!black, fill=green!60!black,inner sep=1.6pt] (RR10) at (2.5,-10) {};
\node[circle,draw=green!60!black, fill=green!60!black,inner sep=1.6pt] (RR10b) at (3.5,-10) {};
\node[circle,draw=green!60!black, fill=green!60!black,inner sep=1.6pt] (RR11) at (1.5,-11) {};
\node[circle,draw=green!60!black, fill=green!60!black,inner sep=1.6pt] (RR11b) at (2.5,-11) {};
\draw[thick,red,decorate,decoration={snake,amplitude=0.1mm, segment length=1mm}] (L1)--(L2)--(L3);
\draw[thick,red,decorate,decoration={snake,amplitude=0.1mm, segment length=1mm}] (L4)--(L5)--(L6);
\draw[thick,red,decorate,decoration={snake,amplitude=0.1mm, segment length=1mm}] (L9)--(L10)--(L11);

\node (L7-) at (-.5,-6){};
\node (L7--) at (1.5,-6){};
\node (L7---) at (2.5,-6){};
\draw[thick,red,decorate,decoration={snake,amplitude=0.1mm, segment length=1mm}] (L8)--(L7)--(L7-);
\draw[thick,decorate,decoration={snake,amplitude=0.1mm, segment length=1mm}] (R1)--(RR1)--(RR1b);
\draw[thick,decorate,decoration={snake,amplitude=0.1mm, segment length=1mm}] (R2)--(RR2)--(RR2b);
\draw[thick,decorate,decoration={snake,amplitude=0.1mm, segment length=1mm}] (R3)--(RR3)--(RR3b);
\draw[thick,decorate,decoration={snake,amplitude=0.1mm, segment length=1mm}] (R4)--(RR4)--(RR4b);
\draw[thick,decorate,decoration={snake,amplitude=0.1mm, segment length=1mm}] (R5)--(RR5)--(RR5b);
\draw[thick,decorate,decoration={snake,amplitude=0.1mm, segment length=1mm}] (R6)--(RR6)--(RR6b);
\draw[thick,decorate,decoration={snake,amplitude=0.1mm, segment length=1mm}] (R7)--(RR7)--(RR7b);
\draw[thick,decorate,decoration={snake,amplitude=0.1mm, segment length=1mm}] (R8)--(RR8)--(RR8b);
\draw[thick,decorate,decoration={snake,amplitude=0.1mm, segment length=1mm}] (R9)--(RR9)--(RR9b);
\draw[thick,decorate,decoration={snake,amplitude=0.1mm, segment length=1mm}] (R11)--(RR10)--(RR10b);
\draw[thick,decorate,decoration={snake,amplitude=0.1mm, segment length=1mm}] (R12)--(RR11)--(RR11b);

\draw[thick,red,decorate,decoration={snake,amplitude=0.1mm, segment length=1mm}] (L12)--(RR1)--(RR2);
\draw[thick,red,decorate,decoration={snake,amplitude=0.1mm, segment length=1mm}] (L13)--(RR1b)--(RR3);
\draw[thick,red,decorate,decoration={snake,amplitude=0.1mm, segment length=1mm}] (RR2b)--(RR3b)--(RR4);
\draw[thick,red,decorate,decoration={snake,amplitude=0.1mm, segment length=1mm}] (RR4b)--(RR5)--(RR6);
\draw[thick,red,decorate,decoration={snake,amplitude=0.1mm, segment length=1mm}] (RR5b)--(RR6b)--(L7--);
\draw[thick,red,decorate,decoration={snake,amplitude=0.1mm, segment length=1mm}] (L14) --(RR11b)--(RR10b);

\draw[thick,red,decorate,decoration={snake,amplitude=0.1mm, segment length=1mm}] (RR11)--(RR10)--(RR9b);
\draw[thick,red,decorate,decoration={snake,amplitude=0.1mm, segment length=1mm}] (RR9)--(RR8b)--(RR7b);
\draw[thick,red,decorate,decoration={snake,amplitude=0.1mm, segment length=1mm}] (RR8) --(RR7)--(L7---);
\draw[thick,color=yellow!70!black,decorate,decoration={snake,amplitude=0.1mm, segment length=1mm}] (L14) --(R12)--(R11);
\end{tikzpicture}
\caption{The hypergraph when $q\equiv 2\pmod 3$}\label{figC}
\end{figure}
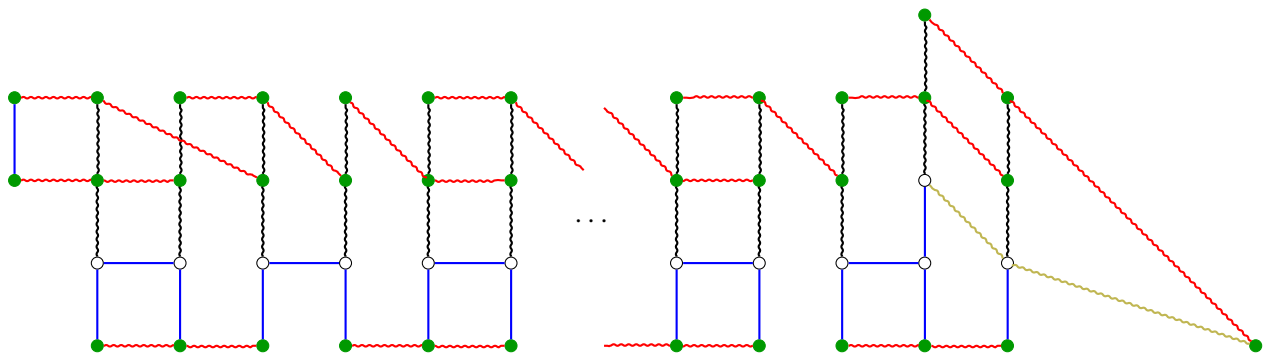

\end{document}